\documentclass[12pt]{amsart}
\usepackage{amsmath,amsthm,amssymb,amsfonts,calligra}
\usepackage[colorlinks=false]{hyperref} 
\usepackage{tikz, pgfplots}
\usepackage{tikz-cd} 
\usepackage[a4paper]{geometry}
\usepackage{bm} 
\usepackage{mathrsfs} 
\usepackage{enumitem} 
\usepackage{stmaryrd}
\usepackage{footnote} 
\makesavenoteenv{tabular}
\usepackage{url}
\usepackage{xcolor}

\mathchardef\mhyphen="2D

\theoremstyle{plain}
\newtheorem{thm}{Theorem}[subsection]
\newtheorem{prop}[thm]{Proposition}
\newtheorem{lem}[thm]{Lemma}
\newtheorem{cor}[thm]{Corollary}

\theoremstyle{definition}
\newtheorem{dfn}[thm]{Definition}

\newtheorem*{question*}{Question}

\theoremstyle{remark}
\newtheorem{rem}[thm]{Remark}

\tikzset{
	labls/.style={anchor=south, rotate=90, inner sep=.5mm}
}
\tikzset{
	labln/.style={anchor=south, rotate=270, inner sep=.5mm}
}

\DeclareMathOperator{\Aut}{Aut}

\DeclareMathAlphabet{\mathcalligra}{T1}{calligra}{m}{n}

\newcommand{\Q}{\mathbb{Q}}
\newcommand{\C}{\mathbb{C}}
\newcommand{\R}{\mathbb{R}}
\newcommand{\Z}{\mathbb{Z}}

\newcommand{\Oo}{\mathcal{O}}

\newcommand{\Lc}{\mathcal{L}}

\newcommand{\ad}{\mathrm{ad}}

\newcommand{\X}{\mathcal{X}} 
\newcommand{\Xo}{\mathcal{X}^\circ}

\newcommand{\cu}{\Sigma} 
\newcommand{\cuo}{\Sigma^\circ} 
\newcommand{\Bb}{\mathbf{B}} 
\newcommand{\Bbo}{\mathbf{B}^\circ} 
 
\newcommand{\discr}{\mathrm{discr}} 
\newcommand{\F}{\mathcal{F}}

\newcommand{\Xt}{\tilde{\mathcal{X}}}

\newcommand{\pib}{\bm{\pi}} 
\newcommand{\Hig}{\mathbf{Higgs}}

\newcommand{\btcu}{\bm{\tilde{\Sigma}}} 
\newcommand{\tcu}{\tilde{\Sigma}} 
\newcommand{\bLambda}{\bm{\Lambda}} 
\newcommand{\bsigma}{\bm{\sigma}} 
\newcommand{\Tc}{\mathcal{T}} 
\newcommand{\Tcb}{\overline{\mathcal{T}}} 
\newcommand{\Tco}{\mathcal{T}^\circ} 
\newcommand{\gfr}{\mathfrak{g}}

\newcommand{\tfr}{\mathfrak{t}} 

\newcommand{\Ub}{\bm{U}} 
\newcommand{\Ubt}{\bm{\tilde{U}}}

\newcommand{\Hit}{\bm{h}}

\newcommand{\VH}{\mathsf{V}}

\newcommand{\Db}{\bm{D}}

\newcommand{\Cc}{\mathbf{C}} 
 
\newcommand{\ADE}{\mathrm{ADE}}
\newcommand{\BCFG}{\mathrm{BCFG}}
\newcommand{\btSigma}{\bm{\tilde{\Sigma}}}
\newcommand{\tSigma}{\tilde{\Sigma}}

\newcommand{\Sc}{\mathcal{S}}

\newcommand{\tf}{\mathrm{tf}}

\begin{document}
\title{Hitchin and Calabi--Yau integrable systems via variations of Hodge structures}
\author{Florian Beck}
\date{\today}
\address{FB Mathematik, Universit\"at Hamburg, Bundesstrasse 55, 20146 Hamburg, Germany}
\email{florian.beck@uni-hamburg.de}

\maketitle
\begin{abstract}
Since its discovery by Hitchin in 1987, $G$-Hitchin systems for a reducitve complex Lie group $G$ have extensively been studied. 
For example, the generic fibers are nowadays well-understood. 
In this paper, we show that the smooth parts of $G$-Hitchin systems for a simple adjoint complex Lie group $G$ are isomorphic to non-compact Calabi--Yau integrable systems extending results by Diaconescu--Donagi--Pantev.
Moreover, we explain how Langlands duality for Hitchin systems is related to Poincar\'{e}--Verdier duality of the corresponding families of quasi-projective Calabi--Yau threefolds. 
Even though the statement is holomorphic-symplectic, our proof is Hodge-theoretic. 
It is based on polarizable variations of Hodge structures that admit so-called abstract Seiberg--Witten differentials. 
These ensure that the associated Jacobian fibration is an algebraic integrable system.
%
\end{abstract}

\section{Introduction}
An \emph{algebraic integrable system} is defined as a proper morphism $\pi:(M,\omega)\to B$ between a holomorphic symplectic manifold $(M,\omega)$ and a complex manifold $B$ with the following property: 
There exists a Zariski-open and dense subset $B^\circ\subset B$ such that the restriction 
\begin{equation*}
\pi^\circ:=\pi_{|M^\circ}:M^\circ \to B^\circ, \quad M^\circ:=\pi^{-1}(B^\circ)
\end{equation*}
admits a relative polarization and has connected Lagragian fibers. 
The complex version of the Arnold--Liouville theorem implies that these fibers are torsors for abelian varieties. 
%

Two very important and prominent example classes are Hitchin systems and Calabi--Yau integrable systems. 
The total space of the former is the moduli space $\Hig(\cu,G)$ of semistable $G$-Higgs bundles on a compact Riemann surface $\cu$ of genus $\geq 2$ of degree $0$ where $G$ is any simple (more generally reductive) complex Lie group. 
The Hitchin map\footnote{Here $d_j=e_j+1$ for the exponents $e_j$ of the Lie algebra of $G$.}
\begin{equation*}
\Hit_G:\Hig(\cu, G)\to \Bb(\cu,G)\cong \bigoplus_{j=1}^r H^0(\cu, K_\cu^{\otimes d_j})
\end{equation*}
which is a global version of the adjoint quotient\footnote{Here $\gfr=\mathrm{Lie}(G)$ is the Lie algebra of $G$ and $\tfr\subset \gfr$ a Cartan subalgebra with Weyl group $W$.} $\gfr\to \tfr/W$, endows $\Hig(\cu, G)$ with the structure of an algebraic integrable system. 
It has been and still is extensively studied, see \cite{Hit1}, \cite{Hit2}, \cite{Faltings}, \cite{Donagi} for the pioneering articles .
One cornerstone of $G$-Hitchin systems is the determination of the generic fibers up to isomorphism, ultimately settled by \cite{DG} and based on the earlier works \cite{Donagi2}, \cite{Sco1}. 
The isomorphism classes of the generic fibers depend sublty on the type of $G$, whereas the isogeny classes only depends on the Lie algebra of $G$. 

Another cornerstone is Langlands duality for $G$-Hitchin systems (\cite{HauselThaddeus},\cite{DP}).
It implies an ismorphism of the Hitchin bases $\Bb(\cu, G)\cong \Bb(\cu, ^L G )$ for the Langlands dual group $^L G$ of $G$ together with an isomorphism 
\begin{equation}\label{eq:langlandsintro}
\Hig^\circ(\cu, G)^\vee \cong \Hig^\circ(\cu, ^L G)
\end{equation}
of smooth algebraic integrable systems over a Zariski-open and dense $\Bbo(\cu,G)\subset \Bb(\cu,G)$.
Here the superindex $^\vee$ stands for the dual torus fibration. 

The origin of Calabi--Yau integrable systems is in contrast very different: 
Their total spaces are the total spaces of the intermediate Jacobian fibrations $J^2(\X) \to B$ (up to a base change) of smooth families $\X\to B$ of compact or quasi-projective Calabi--Yau threefolds satisfying certain conditions. 
The fibers are by definition Griffiths' intermediate Jacobians
\begin{equation*}
J^2(X_b)=H^3(X_b,\C)/\left(F^2H^3(X_b,\C)+H^3(X_b,\Z)\right), \quad b\in B
\end{equation*}
of the members $X_b$ of the family. 

There is a discrepancy between compact and non-compact Calabi--Yau threefolds: 
First of all, $J^2(X)$ is \emph{not} an abelian variety if $X$ is a non-rigid, compact Calabi--Yau threefolds\footnote{Here a compact Calabi--Yau threefold is by definition a complex threefold with trivializable canonical bundle and $H^{1,0}(X)=0$.} but is a non-degenerate complex torus of index $1$. 
Moreover, $J^2(\X)\to B$ carries the structure of a complex integrable system (up to a base change) if the family $\X\to B$ is a \emph{complete} family of compact Calabi--Yau threefolds, see \cite[Section 2]{DM1}. 

On the other hand, the intermediate Jacobians $J^2(X)$ of quasi-projective Calabi--Yau threefolds\footnote{A quasi-projective Calabi--Yau threefold is a smooth quasi-projective threefold with trivializable canonical bundle.} $X$ are more `flexible': 
Depending on the mixed Hodge structure on $H^3(X,\Z)$, $J^2(X)$ might not even be a generalized complex torus or it might be an abelian variety, cf. \cite[Section 2.b.]{Carlson}. 
However, even if all $J^2(X_b)$ are abelian varieties for a non-trivial family $\X\to B$ of smooth quasi-projective Calabi--Yau threefolds, there are no uniform conditions on $\X\to B$ which guarantee the existence of the structure of an algebraic integrable system on $J^2(\X)\to B$.
One reason is that the deformation theory of quasi-projective Calabi--Yau threefolds is in general more complicated than in the compact case. 

The major goal of this article and its companion \cite{Beck2}, is to deepen the understanding of the relationship between $G$-Hitchin systems and Calabi--Yau integrable systems. 
Such a relationship has first been observed in \cite{DDD} and \cite{DDP} if $G$ is a simple adjoint complex Lie group of type $\mathrm{A}_1$ and $\mathrm{ADE}$ respectively. 
Our first main result extends this relationship to \emph{all} simple adjoint complex Lie groups:

\begin{thm}[= Theorem \ref{thm:main}]\label{Thm1}
Let $\Delta$ be any irreducible Dynkin diagram and $G=G_{ad}(\Delta)$ the corresponding simple adjoint complex Lie group. 
Further let $(\Delta_h,\Cc)$ be the unique pair consisting of an $\ADE$-Dynkin diagram such that $\Delta=\Delta_{h,\Cc}$ for a subgroup $\Cc\subset \Aut(\Delta_h)$. 
Then there exists a family $\bm{\pi}\colon\X\to \Bb(\cu,G)$ of quasi-projective Gorenstein threefolds \emph{endowed with a $\Cc$-action and $\Cc$-trivial canonical classes} satisfying the following:
Over a Zariski-open and dense subset $\Bb^\circ\subset \Bb(\cu,G)$ there is an isomorphism 
\begin{equation}\label{IsoIS1}
\begin{tikzcd}
J^2_{\Cc}(\X^\circ)\arrow[r, "\cong"] & \Hig^\circ(\cu, G)
\end{tikzcd}
\end{equation}
of algebraic integrable systems over $\Bbo$. 
Here $J^2_{\Cc}(\X^\circ)\subset J^2(\X^\circ)$ is determined by the $\Cc$-invariants in cohomology. 
\end{thm}
The procedure to go from a pair $(\Delta_h, \Cc)$ of an irreducible $\ADE$-Dynkin diagram $\Delta_h$ and a non-trivial subgroup $\Cc\subset \Aut(\Delta_h)$ to the irreducible Dynkin diagram $\Delta=\Delta_{h,\Cc}$ is known as \emph{folding} in Lie theory (\cite[Chapter 10.3.]{Springer}, \cite[Chapter 6]{Slo}).
The subscript stands for taking $\Cc$-coinvariants in the root system corresponding to $\Delta_h$, see Appendix \ref{SectFolding} for details. 
The dual process of taking $\Cc$-invariants is depicted as follows explaining the name folding: 
\begin{figure}[h]
\scalebox{0.85}{
\begin{tikzpicture}
\coordinate (A) at (0,0) {}; 
\coordinate (B) at (1,0) {}; 
\coordinate (C) at (2,0) {}; 
\coordinate (D) at (3,0) {}; 
\coordinate (E) at (4,0) {}; 
\coordinate (F) at (5,0) {}; 
\coordinate (G) at (6,0) {};
\coordinate (H) at (7,0) {};
\coordinate (I) at (8,0) {}; 
\coordinate (J) at (9,0) {};
\coordinate (A5) at (-1.5,.25) {}; 
\coordinate (Cc) at (-1.5,-.25) {}; 
\coordinate (B4) at (10.5,0) {}; 
\coordinate (Bhalf) at (1.5, -0.9) {}; 
\coordinate (Hhalf) at (1.5,-1.45) {}; 

\node [fill=black, circle, inner sep=0pt, minimum size=5pt] at (A) {}; 
\node [fill=black, circle, inner sep=0pt, minimum size=5pt] at (B) {};
\node [fill=black, circle, inner sep=0pt, minimum size=5pt] at (C) {};
\node [fill=black, circle, inner sep=0pt, minimum size=5pt] at (D) {};
\node [fill=black, circle, inner sep=0pt, minimum size=5pt] at (E) {}; 
\node [fill=black, circle, inner sep=0pt, minimum size=5pt] at (H) {}; 
\node [fill=black, circle, inner sep=0pt, minimum size=5pt] at (I) {}; 
\node [fill=black, circle, inner sep=0pt, minimum size=5pt] at (J) {}; 
\node at (A5) {$\Delta_h=\mathrm{A}_5$}; 
\node at (B4) {$\Delta=\mathrm{C}_3$}; 
\node at (Cc) {$\Cc=\Z/2\Z$}; 

\draw (A) -- (B);
\draw (B) -- (C); 
\draw (C) -- (D); 
\draw (C) -- (E); 
\draw [<->, shorten <= 0.15cm, shorten >=0.15cm] (A) to [in=90, out=90]  (E);
\draw [<->, shorten <= 0.15cm, shorten >=0.15cm] (B) to [in=90, out=90] (D);
\draw (H) -- (I);
\draw[<->] (F) to (G); 

\draw(8,-0.04) -- (9,-0.04);
\draw(8,0.04) -- (9,0.04);

\draw
(8.6,0) --++  (135:.2) 
(8.6,0)  --++  (-135:.2);

\end{tikzpicture}}
\caption{Folding of $\Delta_h=\mathrm{A}_5$ to $\Delta_{h}^\Cc=\Delta=\mathrm{B}_3$.}
\label{Figure1}
\end{figure}
\\
We emphasize that our approach works for all irreducible Dykin diagrams at the same time and employs Slodowy slices (\cite{Slo}).  
In \cite{Beck2} we show that taking $\Cc$-invariants is the same as working with the global orbifold stacks $[\X/\Cc]\to \Bb$, see Remark \ref{rem:orbifolds}. 

Another goal of this article is to explain deeper aspects of $G$-Hitchin systems through Calabi--Yau integrable systems. 
We achieve this by recovering the Langlands duality statement (\ref{eq:langlandsintro}) as Poincar\'{e}--Verdier duality for the corresponding family $\Xo\to \Bbo(\cu,G)$. 
More precisely, we construct the algebraic integrable system
\begin{equation*} 
J_2^\Cc(\Xo)\to \Bbo(\cu,G)
\end{equation*} 
which is defined by taking $\Cc$-coinvariants in the homology/compactly supported intermediate Jacobian fibration $J_2(\Xo)\to \Bbo(\cu,G)$.
If $G\neq SO(2n+1,\C)$, i.e. the Dynkin diagram of the simple adjoint complex Lie group is not $\mathrm{B}_n$, then we prove in Section 6.2 that Poincar\'{e}--Verdier duality induces the isomorphism 
\begin{equation}\label{eq:isoPV}
J_2^\Cc(\Xo)\cong J^2_\Cc(\Xo)^\vee
\end{equation}
of algebraic integrable systems. 
This is equivalent to the Langlands duality isomorphism (\ref{eq:langlandsintro}), in particular $J_2^\Cc(\Xo)$ is isomorphic to $\Hig^\circ(\cu, ^L G)$ over $\Bbo$ if $^L G\neq Sp(n,\C)$. 
 
The exceptional case $G=SO(2n+1,\C)$, hence $^L G=Sp(n,\C)$, is not too surprising: 
The determination of the generic fibers of $Sp(n,\C)$-Hitchin systems is more subtle than for the other cases, see \cite[Remark 4.2]{DG}, \cite[Section 3]{DP}. 
However, we show that $J_2^\Cc(X_b)$, $b\in \Bbo(\cu, Sp(n,\C))$, is isomorphic to an abelian variety $P^\circ(b)$ which is isogenous to $P(b)\cong \Hit_{Sp(n,\C)}^{-1}(b)$. 
It plays a crucial role in determining the precise isomorphism class of $P(b)$. 
Hence we give a geometric meaning to the algebraic integrable system with generic fibers $P^\circ(b)$ if $^LG=Sp(n,\C)$. 
Note that as a corollary, we understand how $G$-Hitchin systems for any simple \emph{simply-connected} complex Lie group is related to non-compact (homology) Calabi--Yau integrable systems. 

Even though our main results are holomorphic-symplectic statements, our proofs are mainly Hodge-theoretic. 
To motivate this, observe that the smooth part $\pi^\circ:M^\circ\to B^\circ$ of an algebraic integrable system is a family of abelian varieties if it admits a section. 
Such a family is uniquely determined by the associated polarizable integral variation of Hodge structures ($\Z$-VHS) $\VH(\pi^\circ)$ of weight $1$ on $B^\circ$.  
We prove in Proposition \ref{CorAbstractSWdifferential} a partial converse of this procedure: 
If a polarizable $\Z$-VHS $\VH$ of weight $1$ with Gau\ss-Manin connection $\nabla$ admits a global section $\lambda\in H^0(B^\circ, \VH)$ such that 
\begin{equation*}
TB^\circ \to F^1\VH, \quad v\mapsto \nabla_v \lambda, 
\end{equation*}
is an isomorphism, then the family $J(\VH)=\VH_{\Oo}/(F^1+\VH_{\Z})\to B^\circ$ of abelian varieties has the structure of an algebraic integrable system. 
A section $\lambda$ with this property is called an \emph{abstract Seiberg--Witten differential}, see Section \ref{SectIntSys} for an explanation of this terminology. 
From there we proceed as follows:
\begin{enumerate}[label=\roman*)]
\item \label{i}
Determine the $\Z$-VHS $\VH_G$ of weight $1$ on $\Bbo(\cu, G)$ induced by $\Hit_G^\circ:\Hig^\circ(\cu, G) \to \Bbo$ for a simple adjoint or simply-connected complex complex Lie group $G$.
Moreover, we show that it admits an abstract Seiberg--Witten differential $\bm{\lambda}$ which determines $\Hit_G^\circ$ as an algebraic integrable system.  
This is achieved in Section 3.6.

\item \label{ii}
Construct an abstract Seiberg--Witten differential $\rho$ for the $\Z$-VHS $\VH_\Cc^{CY}$ with underlying $\Z$-local system $(R^3\bm{\pi}^\circ\Z)^\Cc_{\tf}$. 
Here $\bm{\pi}^\circ:\Xo\to \Bbo(\cu,G)$ is a family as in Theorem \ref{Thm1}. 
In particular $J_\Cc^2(\Xo)=J(\VH_{\Cc}^{CY})\to \Bbo(\cu,G)$ is an algebraic integrable system, see Section \ref{SectNcCY3s}).

The proofs of the properties of $\X\to\Bb(\cu,G)$ are contained in \cite{Beck2}.
They are not needed to understand this article. 
\item \label{iii}
Construct an isomorphism $\Psi:\VH_G\to \VH^{CY}$ of $\Z$-VHS and show that it satisfies $\Psi(\bm{\lambda})=\rho$. 
In particular, it implies the isomorphism (\ref{IsoIS1}).
This is achieved in Section \ref{SectIsoHit} by using the Leray spectral sequence and a careful study how cohomology interacts with $\Cc$-invariants. 
\end{enumerate}
Let us comment on the not yet mentioned sections: 
In Section 2 we fix the notation of variations of Hodge of weight $1$ and introduce abstract Seiberg--Witten differentials. 
The first parts of Section 3 review generic Hitchin fibers and their Hodge structures focusing on the case of simple adjoint and simply-connected complex Lie groups. 
In these cases some of the general arguments simplify so that we give an almost self-contained account.  
Section 3.5 shows that the corresponding Hodge structures are determined by Zucker's Hodge structures (\cite{Zucker}).
This is crucial and eventually guarantees that the isomorphism $\Psi$ in \ref{iii} respects Hodge filtrations. 
The final Section 6 explains the relation between Poincar\'{e}--Verdier and Langlands duality.

\subsection*{Acknowledgements}
It is a pleasure to thank my advisors Katrin Wendland and Emanuel Scheidegger for suggesting this topic to me and all their support. 
Special thanks go to Ron Donagi and Tony Pantev.
Finally, I thank the anonymous referee for very helpful comments. 
This work was supported by the DFG Graduiertenkolleg 1821 "Cohomological Methods in Geometry" and through the DFG Emmy-Noether grant on "Building blocks of physical theories from the geometry of quantization and BPS states", number AL 1407/2-1.

\section{Abstract Seiberg--Witten differentials}\label{SectIntSys}
We fix our notation for variations of Hodge structures by recalling the correspondence between families of abelian varieties and polarizable integral variations of Hodge structures ($\Z$-VHS) of weight $1$. 
A $\Z$-VHS of weight $1$ is denoted by $\VH=(\VH_{\Z},F^\bullet \VH_{\Oo})$ where $\VH_{\Oo}=\VH_{\Z}\otimes_{\Z}\Oo_B$ and $F^\bullet$ is the Hodge filtration. 
The dual of $\VH$ is defined by 
\begin{equation*}
\VH^\vee:=\mathrm{Hom}_{\mathrm{VHS}}(\VH,\Z_B)(-1).
\end{equation*}
The Tate twist $(-1)$ guarantees that $\VH^\vee$ is a $\Z$-VHS of weight $1$.

To any polarizable $\Z$-VHS $\VH$ of weight $1$ we associate the family 
\begin{equation*}
J(\VH):=\mathrm{tot}\left(\VH_{\Oo}/(F^1\VH_{\Oo}+\VH_{\Z})\right)\to B
\end{equation*}
of polarizable abelian varieties, the (intermediate) Jacobians of $\VH$, over $B$. 
Here $\mathrm{tot}$ stands for the total space of the corresponding fiber bundle. 

Conversely, if $\pi:\mathcal{A}\to B$ is any family of polarizable abelian varieties, then polarizable $\Z$-VHS $\VH(\pi)$ of weight $1$ with $\VH(\pi)_{\Z}=R^1\pi_*\Z$ satisfies $J(\VH(\pi))\cong \mathcal{A}$ over $B$. 
The dual family $\pi^\vee:\mathcal{A}^\vee\to B$ is determined by $\VH(\pi)^\vee$.
Both $\pi$ and $\pi^\vee$ are fiberwise isogenous to each other. 
On the level of $\Z$-VHS of weight $1$ this corresponds to the following: 
Two polarizable $\Z$-VHS $\VH, \VH'$ of weight $1$ isogenous to each other, $\VH \sim \VH'$, if $J(\VH)$ is fiberwise isogenous to $J(\VH')$ or equivalently $\VH\otimes \Q\cong \VH'\otimes \Q$ as polarizable $\Q$-VHS of weight $1$. 

Of particular interest for us are families of abelian varieties which come from algebraic integrable systems. 
\begin{dfn}\label{DefnIntSys}
	Let $(M,\omega)$ be a holomorphic symplectic manifold and $B$ a complex manifold.
	An algebraic integrable system is a proper holomorphic map $\pi:(M,\omega)\to B$ with the following property: There is a Zariski-open dense subset $B^\circ\subset B$ such that the restriction 
	\begin{equation*}
	\pi^\circ:=\pi_{M^\circ}:M^\circ \to B^\circ,\quad M^\circ=\pi^{-1}(B^\circ),
	\end{equation*}
	has connected Lagrangian fibers and admits a relative polarization. 
\end{dfn}
It follows as in the $C^\infty$-case (the Arnold-Liouville theorem, see \cite[Chapter IV]{GS}) that the fibers of $\pi$ are torsors for abelian varieties. 
More precisely, $\pi^\circ$ is a torsor for the family $J(\VH(\pi^\circ))$ of polarizable abelian varieties over $B^\circ$.  
In particular, if $\pi^\circ$ admits a section, then $M^\circ \cong J(\VH(\pi^\circ))$ over $B^\circ$. 

The holomorphic sympelctic form $\omega$ further induces an isomorphism 
\begin{equation}\label{eq:iota}
\iota: \VH_{\Oo}/F^1\to T^*B.
\end{equation}
Indeed, $\omega$ induces the isomorphism $\ker d\pi^\circ\cong (\pi^\circ)^*T^*B^\circ,$ $v\mapsto \omega(v,-)$. 
The adjunction formula yields $\pi_*^\circ(\pi^\circ)^*T^*B^\circ\cong T^*B$ because the fibers of $\pi^\circ$ are connected. 
Moreover, $\pi_*^\circ \ker d\pi^\circ\cong \VH_{\Oo}/F^1$ by the previous discussion.
These isomorphisms combine to give (\ref{eq:iota}). 

We now address the converse and give a sufficient condition on a polarizable $\Z$-VHS $\VH$ of weight $1$ such that $J(\VH)\to B$ carries the structure of an algebraic integrable system.  
\begin{prop}\label{CorAbstractSWdifferential}
	Let $\VH$ be a $\Z$-VHS of weight $1$ over $B$, $Q$ a polarization on $\VH_\R=\VH_{\Z}\otimes_\Z \R$ and $\nabla$ the Gau\ss -Manin connection on $\VH_{\Oo}$.
	Assume there is a global section $\lambda\in \Gamma(B,\VH_{\Oo})$ such that 
	\begin{equation}\label{IsoF1VHS}
	\phi_\lambda:TB\to F^1\VH, \quad X\mapsto \nabla_X\lambda,
	\end{equation}
	is an isomorphism. 
	Further let $\iota:\VH_{\Oo}/F^1\to T^*B$ be the isomorphism induced by (\ref{IsoF1VHS}) and the polarization $Q$.
	Then the family  
	$
	J(\VH) \to B
	$
	of abelian varieties carries a unique Lagrangian structure\footnote{A Lagrangian structure on a map $\pi:M\to B$ is a holomorphic symplectic structure $\omega$ on $M$ such the smooth part of the fibers of $\pi$ are Lagrangian with respect to $\omega$.} $\omega_\lambda$ which makes the zero section Lagrangian and induces $\iota$. 
	It is independent of $Q$ up to symplectomorphisms. 
	Moreover, the same results hold true for $J(\VH')\to B$ where $\VH'$ is any VHS in the isogeny class of $\VH$, in particular for $\VH'=\VH^\vee$.  
\end{prop}
\begin{proof}
	We begin by recalling how $\iota: \VH_{\Oo}/F^1\to T^*B$ is constructed.
	To this end, observe that the polarization $Q$ induces an isomorphism 
	\begin{equation*}
	\phi_Q:\VH_{\Oo}/F^1\to (F^1)^*.
	\end{equation*}
	Then $\iota$ is the composition $\iota=\phi_\lambda^\vee\circ\phi_Q$.
	These isomorphisms further induce isomorphisms (denoted by the same symbols)\footnote{If $Q$ is not defined over $\Z$, then $\phi_Q(\VH_\Z)$ is \emph{not} contained in $\VH^\vee_\Z=Hom(\VH_\Z,\Z)$. In any case, $\phi_Q(\VH_\Z)$ is a local system of lattices.}
	\begin{equation*}
	\begin{tikzcd}
	J(\VH)\ar[r, "\phi_Q"] & (F^1)^*/\phi_Q(\VH_\Z) \ar[r, "\phi_\lambda^\vee"] & T^*B/\Gamma,\quad \Gamma:= \phi_\lambda^*(\phi_Q(\VH_\Z)).
	\end{tikzcd}
	\end{equation*}
	If we show that $\Gamma\subset T^*B$ is Lagrangian, then the canonical symplectic structure $\eta$ on $T^*B$ descends to a symplectic structure $\hat{\eta}$ on $T^*B/\Gamma$. 
	The induced symplectic structure on $J(\VH)$ will satisfy all the claimed properties, in particular the zero section will be Lagrangian. 
	\\
	\\
	To show that $\Gamma\subset T^*B$ is Lagrangian, we have to prove that the image of $\VH_{\Z}$ in $T^*B$ under $\iota$ consists of closed (local) $1$-forms.
	If $\gamma$ is a local section of $\phi_Q(\VH_\Z)\subset (F^1)^*$, then its image is the local $1$-from 
	\begin{equation*}
	\phi_\lambda^\vee(\gamma)(X)=\langle \gamma, \nabla_X \lambda \rangle, \quad X\in TB,
	\end{equation*} 
	where the brackets are the duality pairing between $(F^1)^*$ and $F^1$. 
	Since $\VH^\vee_\Z$ is a local system and $\nabla$ is flat, we can represent $\gamma$ around $b\in U\subset B$ as some fixed element $\gamma_0\in \phi_\lambda^\vee(\VH_{\Z})_b$, $\nabla$ as $d$ and $\lambda$ as a map $f:U\to \VH_b$. 
	In particular, $v=\nabla_X \lambda\in \VH$ is represented by $df(X)$ where $X\in TU$. 
	It then follows that $g:U\to \C$, $g(b)=\langle \gamma_0,f(b)\rangle$, satisfies 
	\begin{align*}
	dg(X)=&\frac{d}{dt}_{|t=0} g(\alpha(t)) \\
	=&\langle \gamma_0, df(X)\rangle \\
	=&\phi_{\lambda}^\vee(\gamma)(v).
	\end{align*}
	Here $\alpha$ is a curve representing the tangent vector $X$.    
	Hence $\phi_\lambda^\vee(\gamma)$ is locally exact and therefore closed. \\ 
	Now let $Q'$ be another polarization. 
	Then the previous construction can be performed for $Q'$ as well and we denote by $\omega$ and $\omega'$ the corresponding Lagrangian structures. 
	Morever, it follows that there is an automorphism $\psi: T^*B\to T^*B$ such that $\psi(\Gamma)=\Gamma'=\phi_\lambda^\vee\circ \phi_{Q'}(\VH_\Z)$. 
	It induces a symplectomorphism $\psi:(T^*B/\Gamma, \hat{\eta})\to (T^*B/\Gamma',\hat{\eta}')$.
	Since $\omega$ and $\omega'$ on $J(\VH)$ are pull backs of $\hat{\eta}$ and $\hat{\eta}'$ respectively, it follows that $\omega$ and $\omega'$ are symplectomorphic to each other. \\
	The last statement is immediate because if $\VH\sim \VH'$ are isogenous, then $\VH'$ admits a section $\lambda\in \Gamma(B,\VH_{\Oo}')$ with the same properties as well. 
\end{proof}

\begin{dfn}
	A section $\lambda\in \Gamma(B,\VH_{\Oo})$, such that 
	\begin{equation*}
	TB\to F^1\VH, \quad X\mapsto \nabla_X \lambda, 
	\end{equation*}
	is an isomorphism as above, will be called an \emph{abstract Seiberg--Witten differential}. 
\end{dfn}

\begin{rem}
\begin{enumerate}[label=\alph*)]
\item
The previous definition is motivated by Seiberg--Witten differentials of Seiberg--Witten integrable systems, cf. \cite{DonagiSW}.
\item 
If $(\VH,\lambda)$ is a polarizable $\Z$-VHS of weight $1$, then it is shown in \cite[Section 2.2.4]{Beck-thesis} that $J(\VH)\to B$ satisfies the cubic condition of Donagi--Markman (\cite{DM1}).
This gives an alternative proof of Proposition \ref{CorAbstractSWdifferential}.
\item 
It is not difficult to generalize Proposition \ref{CorAbstractSWdifferential} to $\Z$-VHS $\VH$ of weight $1$ which admit a polarization with index $k>0$, i.e. the associated Hodge metric has index $k$. 
In this case, $(J(\VH),\omega_\lambda)\to B$ is a complex integrable system of index $k$. 
For example, compact Calabi--Yau integrable systems (\cite{DM1}) are examples of complex integrable systems of index $1$. 
We refer the reader to \cite[Section 2]{Beck-notes} for details where it is further shown that compact Calabi--Yau integrable systems admit abstract Seiberg--Witten differentials. 
\end{enumerate}
\end{rem}

\section{VHS of Hitchin systems}\label{SectHitSys}
In this section, we determine the $\Z$-VHS of weight $1$ corresponding to the smooth part of $G$-Hitchin systems.
Moreover, we determine the Lagrangian structure in terms of an abstract Seiberg--Witten differential. 
One observation is that the Hodge filtration is given by Zucker's  (\cite{Zucker}), see Lemma \ref{ApplicationZuckersVHS} which is crucial in Section \ref{ss:isoais}.
In Section \ref{ss:hitchinadjsc}, we specialize to simple adjoint complex Lie groups which are either of \emph{adjoint type} ($G=G_{ad}$) or \emph{simply connected} ($G=G_{sc}$).
These are the relevant cases for us. 
Even though these cases could be extracted from \cite[Section 3]{DP}, we give essentially self-contained proofs for the convenience of the reader. 

\subsection{Hitchin systems}
We briefly recall the construction of $G$-Hitchin systems for any simple complex Lie group $G$ and compact Riemann surface $\cu$ of genus $g(\cu)\geq 2$. 
For a more extensive treatment, see \cite{Hit1}, \cite{Hit2}, \cite{Faltings}, \cite{Donagi}, \cite{DG}, \cite{DP}.

A $G$-Higgs bundle on $(P,\varphi)$ on $\cu$ is a $G$-bundle $P$ on $\cu$ together with a section $\varphi\in H^0(\cu,\mathrm{ad}(P)\otimes K_\cu)$ where $K_\cu$ is the canonical bundle and $\mathrm{ad}(P)$ is the holomorphic vector bundle associated with the adjoint representation of the Lie algebra $\gfr$ of $G$.
The moduli space $\Hig(\cu,G)$ of semistable $G$-Higgs bundles \emph{of degree $0$} on $\cu$ is a quasi-projective variety of dimension $2\dim(G)(g(\cu)-1)$. 
Its smooth locus $\Hig(\cu,G)^{sm}$ carries a holomorphic symplectic form $\omega_H$. 

To construct the $G$-Hitchin system, fix a maximal torus $T\subset G$ with corresponding Cartan subalgebra $\tfr\subset \gfr=\mathrm{Lie}(G)$ of rank $r$, root system $R$ and Weyl group $W$.
If the canonical bundle $K_\cu$ of $\cu$ is considered as a $\C^*$-bundle, define the associated bundles
\begin{align*}
&u:\Ub:=K_\cu\times_{\C^*}\tfr/W\to \cu, \\
&\tilde{u}:\Ubt:=K_{\cu}\times_{\C^*}\tfr\to \cu.
\end{align*}
The adjoint quotient $\chi:\gfr\to \tfr/W$ globalizes to give a well-defined morphism
\begin{equation*}
\Hit=\Hit_G: \Hig(\cu,G)\to H^0(\cu,\Ub),\quad [P,\varphi]\mapsto \chi(\varphi).
\end{equation*}
We abuse notation and denote by $\Hit$ the restriction of $\Hit$ to the smooth locus $\Hig(\cu,G)^{sm}$.
Then the $G$-Hitchin system is given by 
\begin{equation*}
\Hit:(\Hig(\cu,G)^{sm},\omega_H)\to \Bb(\cu,G):=H^0(\cu, \Ub)
\end{equation*}
which is an algebraic integrable system in the sense of Definition \ref{DefnIntSys}. 

To investigate the $\Z$-VHS determined by $\Hit$ over an open dense subset of $\Bb=\Bb(\cu,G)$, we introduce cameral curves. 
Since the $W$-action on $\tfr$ commutes with the natural $\C^*$-action, $\Ubt$ inherits a $W$-action. 
The quotient map $q:\tfr\to \tfr/W$ glues to the morphism $\bm{q}:\Ubt\to \Ub$. 
The universal cameral curve is defined by the cartesian square
\begin{equation}\label{UniversalCameralCurve}
\begin{tikzcd}
\btcu\arrow[r] \arrow[d, "\bm{p}_1"] \arrow[dd, bend right=50, "\bm{p}"'] & \Ubt \arrow[d, "\bm{q}"] \\
\cu\times \Bb \arrow[r, "ev"] \arrow[d, "\bm{p}_2=pr"] & \Ub \\ 
\Bb.
\end{tikzcd}
\end{equation}
By construction $\btcu$ inherits a $W$-action. 
The pullback $\tcu_b:=i_b^*\btcu$ via the inclusion $i_b:\cu\to \{b\}\times \cu$ is the cameral curve $\tcu_b\hookrightarrow \Ubt$ corresponding to $b\in \Bb$. 
We denote by
\begin{equation*}
p_b:\tcu_b\to \cu 
\end{equation*}
the induced projection. 
For generic $b\in \Bb$ these curves are non-singular and $p_b$ is a simply ramified $W$-Galois covering. 
More precisely, let 
\begin{equation}\label{DefBbo}
\Bbo:=\{b\in \Bb~|~ b\text{ transversal to }\discr(\bm{q})^{sm} \}
\end{equation}
where $\discr(\bm{q})^{sm}$ denotes the smooth locus of the discriminant $\discr(\bm{q})$ of $\bm{q}$. 
Then $\Bbo\subset \Bb$ is a non-empty Zariski-open subset and comprises the locus of smooth cameral curves with simple Galois ramification (\cite[Section 1]{Sco1}). 
Moreover, $\Bbo\subset \Bb$ is contained in the smooth locus of $\Hit:\Hig(\cu,G)\to \Bb$ and the restriction satisfies
\begin{equation*}
\Hig^\circ(\cu,G):=\Hit^{-1}(\Bbo)\subset \Hig(\cu, G)^{sm}.
\end{equation*}

\subsection{Stratifications}\label{SectStratifications}
In the following we work with $\Bbo$ rather than the full Hitchin base $\Bb$. 
It is therefore useful to restrict $\Ub$ to a smaller open subset $\Ub^1\subset \Ub$ through which the evaluation map $ev$ factorizes. 
Consider the open subsets 
\begin{align*}
&\tfr^1=\tfr-\bigcup_{\alpha \neq \beta} \tfr_{\alpha}\cap \tfr_{\beta}\subseteq \tfr,  \\
&\tfr^1/W:=q(\tfr^1)\subseteq \tfr/W.
\end{align*}
They are stratified by 
\begin{align*}
&\tfr^1=\tfr^\circ \cup D:=\tfr^\circ\cup \left(\bigcup_{\alpha\in R} \tfr_{\alpha}- \cup_{\beta\neq \gamma} \tfr_{\beta}\cap \tfr_{\gamma}     \right),\\  
&\tfr^1/W=\tfr^\circ/W\cup  (D/W), 
\end{align*}
where $\tfr^\circ\subset \tfr$ are the regular elements. 
Note that $D=\tfr^1-\tfr^\circ$ decomposes into $D=D^s\cup D^l$ where $D^s$ and $D^l$ corresponds to short and long roots respectively. 
Of course, $\tfr^1$ and $\tfr^1/W$ together with their stratifications can be glued to give
\begin{align}
&\Ubt^1=\Ubt^\circ\cup \Db=\Ubt^\circ\cup \Db^s\cup \Db^l, \label{Ubt1} \\
&\Ub^1=\Ub^\circ\cup (\Db/W)=\Ub^\circ \cup (\Db^s/W) \cup (\Db^l/W) \label{Ub1}
\end{align}
and $\discr(\bm{q}_{|\Ubt_1})=\Db$. 
Hence $\Bbo$ are precisely the sections $b:\cu\to \Ub^1$ that intersect $\Db$ transversally. 
Note that every $b\in \Bbo$ necessarily intersects $\Db$ by the compactness of $\cu$. 

The restriction $\bm{p}^\circ:\btSigma^\circ \to \Bbo$ of the universal cameral curve to the preimage of $\Bbo$ factorizes as
$
\bm{p}=\bm{p}^\circ_2\circ \bm{p}^1_1.
$
Here $\bm{p}^\circ_2=pr:\cu\times\Bbo\to \Bbo$ is the projection and $\bm{p}_1^1:\btSigma^\circ\to \cu\times \Bbo$ is obtained by restricting the cartesian square in (\ref{UniversalCameralCurve}).

\subsection{Generic Hitchin fibers}
We next express the fiber $\Hit^{-1}(b)$, $b\in \Bbo$, in terms of the smooth cameral curves $\tcu_b$ following \cite{DG}, \cite{DP}. 
To do so, consider the sheaf 
\begin{equation*}
\Tcb(b):=\Tcb:=p_{b,*}^W(\bLambda\otimes \mathcal{O}_{\tcu_b}^*), 
\end{equation*}
on $\cu$ for $b\in \Bbo$. 
Here
$
\bLambda_G=\bLambda=\mathrm{Hom}(\C^*, T) 
$
is the cocharacer lattice of $G$ and
\begin{equation*}
p_{b,*}^W=(.)^W\circ p_{b,*}
\end{equation*}
is the equivariant direct image for the diagonal $W$-action. 
If $D^\alpha\subset \tcu_b$ is the ramification divisor corresponding to $\alpha \in R$, then we define the subsheaf
\begin{equation*}
\Tc(U):=\{ t\in \Tcb(U)~|~\alpha(t)_{|D^\alpha}=+1~\forall \alpha\in R \}
\end{equation*}
of $\Tcb$.
Here we consider each root $\alpha\in R$ as a homomorphism $\alpha:T\to \C^*$ and use the canonical identification $\bLambda\otimes \C^*= T$. 
Further let $\Tco\subset \Tc$ be the connected component of $\Tc$. 
By definition, these three sheaves satisfy 
\begin{equation*}
\Tco\subset \Tc\subset \Tcb.
\end{equation*} 

It is proven in \cite[Corollary 4.6]{DG} that $\Hit^{-1}(b)$, $b\in \Bbo$, is a torsor for $H^1(\cu, \Tc)$.
To determine the connected component $P$ of $H^1(\cu,\Tc)$ it is useful to consider the connected components $P^\circ$ and $\overline{P}$ of $H^1(\cu,\Tco)$ and $H^1(\cu,\Tcb)$ respectively.
If the point $b\in \Bbo$ and the structure group $G$ is relevant, we write $Q_G(b)$ for $Q=P^\circ,P, \overline{P}$.

\begin{prop}\label{PrymsComplexStr}
Let $b\in \Bbo$. 
Then $P^\circ(b)$, $P(b)$ and $\overline{P}(b)$ are abelian varieties which are all isogenous to each other. 
Any of these abelian varieties $Q$ is in turn isogenous to the abelian variety 
\begin{equation}\label{eq:A}
J(H^1(\tcu_b, \bLambda)^W). 
\end{equation}
In particular, the complex structure of $Q$ is determined by the Hodge filtration $F^\bullet H^1(\tcu, \tfr)^W$ under the canonical identification $\bLambda\otimes \C=\tfr$.
\end{prop}

\begin{proof}
	The first claim already appeared in \cite{DP} but we need to elaborate on their proof. 
	To simplify the presentation, we drop $b$ from the notation. 
	
	Consider the Grothendieck spectral sequence 
	\begin{equation*}
	 R^pa_* R^qp_*^W \F \Rightarrow R^{p+q}\tilde{a}_*^W \F
	\end{equation*}
	for the composition $a_*\circ p_*^W=\tilde{a}^W_*$ where $a:\cu\to pt$ and $\tilde{a}:\tcu\to pt$ are the constant maps. 
	Note that $\tilde{a}^W_*(\F)=(.)^W\circ \tilde{a}_*(\F)=H^0(\tcu, \F)^W$ for any $W$-sheaf $\F$ on $\tcu$. 
	The corresponding five-term exact sequence of this spectral sequence reads as 
	\begin{equation}\label{Seq5TermsAVs}
	\begin{tikzcd}
	0\ar[r] & H^1(\cu, p_*^W\F) \ar[r, "\gamma"] & H^1(\tcu, \F)^W \ar[r]  \arrow[d, phantom, ""{coordinate, name=Z}] & H^0(\cu, R^1p_*^W\F) 
	\arrow[dll, rounded corners, to path={ -- ([xshift=2ex]\tikztostart.east)
		|- (Z) [near end]\tikztonodes
		-| ([xshift=-2ex]\tikztotarget.west) -- (\tikztotarget)}] \\
	& H^2(\cu, p_*^W\F)\ar[r] & H^2(\tcu, \F)^W. 
	\end{tikzcd}
	\end{equation}
	Since $p$ is a finite map, $R^1p_*^W\F$ is isomorphic to $\mathcal{H}^1(W,p_*\F)$ (see \cite[Section 5]{Gro1}).
	The latter sheaf has stalks $H^1(W,(p_*\F)_x)$ which is finite because $H^k(W,M)$ is finite for $k\geq 1$ and any $W$-module $M$. 
	As $\mathcal{H}^1(W,p_*\F)$ is a local system on $\cuo=\cu-Br_b$, $H^0(\cu, R^1p_*^W\F)$ is finite. 
	It follows that (dropping $b$ from the notation)
	\begin{equation*}
	\overline{P}=H^1(\cu, p_*^W\F)^\circ, \quad \F=\bLambda\otimes \mathcal{O}_{\tcu},
	\end{equation*} 
	is an abelian variety. 
	Indeed, it is classical that the connected component of $H^1(\tcu, \bLambda\otimes \mathcal{O}^*_{\tcu})^W$ is the abelian variety $J$ of (\ref{eq:A}).
	Restricting $\gamma$ of (\ref{Seq5TermsAVs}) to the connected components shows that $\gamma^\circ:\overline{P}\to \mathcal{A}$ is injective with finite cokernel, i.e. is an isogeny. 
	In particular, $\overline{P}$ carries the structure of an abelian variety. 
	\\
	To prove the statement for $P^\circ$ and $P$, consider the exact sequences
	\begin{equation*}
	\begin{tikzcd}
	0 \ar[r] & \Tco \ar[r] & \Tc \ar[r] & \Tc/\Tco \ar[r] & 0,
	\end{tikzcd}
	\end{equation*}
	\begin{equation*}
	\begin{tikzcd}
	0 \ar[r] & \Tc \ar[r] & \Tcb \ar[r] & \Tcb/\Tc \ar[r] & 0.
	\end{tikzcd}
	\end{equation*}
	Note that the quotients are supported on the branch locus of $\tcu\to \cu$, i.e. they are (sums of) skyscrapers. 
	The corresponding long exact sequences show that each of the natural maps $H^1(\cu, \Tco)\to H^1(\cu, \Tc)\to H^1(\cu, \Tcb)$ is surjective with finite kernel. 
	Hence the restrictions $P^\circ \to P \to \overline{P}$ are isogenies and are therefore isogenous to $J$.
\end{proof}

\begin{cor}\label{cor:Hodge}
Let $b\in \Bbo$.
The finitely generated abelian groups 
\begin{equation}\label{eq:cocharPoPbar}
H^\circ_{G,\Z}=H^1(\cu, p_{b,*}^W \bLambda_G),\quad \overline{H}_{G,\Z}=H^1(\cu, p_{b,*}^W\bLambda_G^\vee)^\vee
\end{equation}
carry polarizable Hodge structures $H^\circ_G$ and $\overline{H}_G$ of weight $1$ such that there are canonical isomorphisms
$J(H^\circ)\cong P_G^\circ,$ $J(\overline{H})=\overline{P}_G$.
\end{cor}
\begin{proof}
As real tori, the abelian varieties $Q=P_G^\circ,$ $\overline{P}_G$ are determined by $\mathsf{cochar}(Q)\otimes S^1$. 
The cocharacter lattice of $Q$ has been deteremined in \cite[Claim 3.6]{DP} as in (\ref{eq:cocharPoPbar}). 
The complex structure is recovered by Proposition \ref{PrymsComplexStr} which identifies the universal covering of $Q$ with $H^1(\tcu_b, \tfr)^W$. 
In particular, the corresponding Hodge structure $H_Q=(H_{\Z},F^\bullet H_{\C})$ of weight $1$ with $J(H_Q)\cong Q$ is given by 
\begin{equation}
H_{\Z}=\mathsf{cochar}(Q),\quad F^\bullet H_{\C}=F^\bullet H^1(\tcu_b,\tfr)^W.
\end{equation}
\end{proof}

\subsection{Adjoint and simply-connected groups}\label{ss:hitchinadjsc}
Let $G_{ad}=G_{ad}(\Delta)$ and $G_{sc}=G_{sc}(\Delta)$ be the simple adjoint and simply-connected complex Lie group respectively with fixed Dynkin diagram $\Delta$. 
The considerations of the previous section simplify as we explain next. 

The cocharacter lattices $\bLambda_{ad}=\bLambda_{G_{ad}}$ and $\bLambda_{sc}=\bLambda_{G_{sc}}$ satisfy
\begin{equation}\label{eq:adjoint}
\mathsf{coroot}_\gfr=\bLambda_{sc} \subsetneq \bLambda_{ad}=\mathsf{coweight}_\gfr\subset \tfr^\vee
\end{equation}
Here $\mathsf{coroot}_{\gfr}=\langle R^\vee\rangle_{\Z}$ is the coroot lattice and $\mathsf{coweight}_{\gfr}$ is the coweight lattice of $\gfr$.  
The corresponding character lattices $\bLambda_{ad}^\vee$ and $\bLambda_{sc}^\vee$ satisfy
\begin{equation}\label{eq:simplyconnected}
\mathsf{root}_\gfr=\bLambda_{ad}^\vee \subsetneq \bLambda_{sc}^\vee=\mathsf{weight}_{\gfr}\subset \tfr
\end{equation}
We denote by $\langle \bullet , \bullet \rangle: \tfr \times \tfr^\vee\to \C$ the natural pairing.

Since the Hitchin base only depends on the common Lie algebra $\gfr$, we canonically identify $\Bb=\Bb(\cu,G_{ad})=\Bb(\cu,G_{sc})$ and similarly for $\Bbo$, cf. (\ref{DefBbo}).
Let $P^\circ_{ad},$ $P_{ad},$ $\overline{P}_{ad}$ and $P^\circ_{sc},$ $P_{sc}$, $\overline{P}_{sc}$ be the abelian varieties of Lemma \ref{PrymsComplexStr} at $b\in \Bbo$ for the structure groups $G_{ad}$ and $G_{sc}$ respectively. 

\begin{prop}\label{p:cochar}
	With the previous notation there are the following canonical isomorphisms
	\begin{equation}\label{eq:isoPoP}
	P^\circ_{ad}\cong P_{ad}, \quad P_{sc}\cong \overline{P}_{sc}
	\end{equation}
	of abelian varieties. 
	Additionally, there are the special cases
	\begin{equation}\label{eq:allIso}
	\begin{aligned}
	&P_{ad}^\circ= P_{ad}=\overline{P}_{ad}\quad &\mathrm{if } ~~  \Delta\neq \mathrm{B}_k,
	\\
	&P_{sc}^\circ=P_{sc}=\overline{P}_{sc}\quad &\mathrm{if } ~~ \Delta\neq \mathrm{C}_k. 
	\end{aligned}
	\end{equation}
\end{prop}	
\begin{rem}\label{rem:hodgep}
By Corollary \ref{cor:Hodge}, the $\Z$-Hodge structure $H_G$ of weight $1$ with $P_G\cong J(H_G)$ is $H^\circ_G$ or $\overline{H}_G$ in case $G$ is of adjoint type or simply-connected. 
\end{rem}

\begin{proof}
	We first prove (\ref{eq:isoPoP}).
	This follows if $H^1(\cu,\mathcal{T}_{ad})\cong H^1(\cu,\mathcal{T}_{ad}^\circ)$ and $H^1(\cu, \mathcal{T}_{sc})\cong H^1(\cu, \overline{\mathcal{T}}_{sc})$ respectively. 
	
	To this end, it is convenient to introduce the real versions $\mathcal{F}_{\R}$ of $\mathcal{F}=\Tco,$ $\Tc$ and $\overline{\Tc}$ (dropping $ad$ or $sc$ from the notation if unnecessary).
	These are defined by replacing $\mathcal{O}_{\tcu_b}^*$ by the constant sheaf for the circle group $S^1$ in their definition.
	Explicitly, $\Tco_{\R}=(p_{*}^W\bLambda)\otimes S^1$ and $\overline{\Tc}_{\R}=p_{*}^W( \bLambda\otimes S^1)$. 
	As in the complex case, we canonically identify $\bLambda\otimes S^1=T_{\R}$, the compact real torus inside the maximal torus $T$. 
	
	The relevance of the real versions $\mathcal{F}_{\R}$ comes from the fact that the natural morphism $\mathcal{F}_{\R}\hookrightarrow \mathcal{F}$ induces an isomorphism
	\begin{equation}\label{eq:h1realversion}
	H^1(\cu,\mathcal{F}_{\R})\cong H^1(\cu, \mathcal{F}),
	\end{equation}
	see \cite[Lemma 3.2]{DP}.
	Let $s\in \cu$ be a branch point of $p$ and $\alpha$ a root in the $W$-orbit corresponding to $s$. 
	Then the stalks of these sheaves are easily computed to be
	\begin{align*}
	\overline{\Tc}_{\R,s}&=\{ \lambda\otimes z ~|~ \alpha^\vee(z^{\langle \alpha, \lambda \rangle })= 1\in T_{\R} \}, 
	\\
	\Tc_{\R,s}&=\{ \lambda\otimes z ~|~ z^{\langle \alpha, \lambda \rangle }=1 \in S^1  \}, 
	\\ 
	\Tco_{\R,s}&=\{ \lambda\otimes z ~|~\langle \alpha, \lambda \rangle=0 \in \Z  \}.
	\end{align*}
	In the adjoint case, let $T_{\R}=\bLambda\otimes S^1 \ni \lambda\otimes z\neq 1$ with $z^{\langle \alpha, \lambda \rangle}=1$ and denote $k=\langle \alpha, \lambda \rangle\in \Z$. 
	By (\ref{eq:adjoint}), there exists $\lambda_0$ such that $\langle \alpha, \lambda_0\rangle =1$ and $\lambda=k \lambda_0$. 
	Since $z^k=1$ by assumption, we conclude a contradiction
	\begin{equation*}
	\lambda \otimes z= k \lambda_0 \otimes z =\lambda_0 \otimes z^k=1\in T_{\R}.
	\end{equation*}
	Consequently, $\Tco_{ad,\R}=\Tc_{ad,\R}$ so that $P^\circ_{ad}\cong P_{ad}$ as real tori. 
	But both are isogenous to $J(H^1(\tcu_b,\bLambda_{ad})^W)$ by Lemma \ref{PrymsComplexStr} so that $P^\circ_{ad}\cong P_{ad}$ as abelian varieties. 
	
	In the simply connected case, observe that each coroot $\alpha^\vee$, seen as a cocharacter $\alpha^\vee:\C^*\to T\subset G_{sc}$, must be injective: 
	By (\ref{eq:simplyconnected}) there exists a character $\mu\in \bLambda_{sc}^\vee$ such that $\mu\circ \alpha^\vee(t)=t$ for all $t\in \C^*$. 
	Hence $P_{sc}\cong \overline{P}_{sc}$ as real tori by (\ref{eq:h1realversion}). 
	As before, it follows that this is in fact an isomorphism of abelian varieties. 
	
	To prove the first equalities in (\ref{eq:allIso}), note that each coroot is injective, i.e. there exists a root $\beta$ such that $\langle \beta, \alpha^\vee \rangle=1$, if the Dynkin diagram $\Delta\neq \mathrm{B}_k$ so that $\Tc_{ad,\R}=\overline{\Tc}_{ad,\R}$.
	Interchanging roots with coroots, the same argument gives $\Tco_{sc,\R}=\Tc_{sc,\R}$ if $\Delta\neq \mathrm{C}_k$. 
\end{proof}
For the next corollary, let $^L G$ be the Langlands dual of a simple complex Lie group $G$. 
By definition, we have $\bLambda_{[^L G]}=\bLambda_G^\vee,$ $\bLambda_{[^L G]}^\vee=\bLambda_G$. 
As shown in \cite[Theorem A (1)]{DP}, there is an isomorphism $\mathsf{I}:\Bb(\cu, G) \to\Bb(\cu, ^L G)$. It is unique up to multiplication with $\C^*$ and satisfies $\mathsf{I}(\Bbo(\cu,G))=\Bbo(\cu, ^L G)$. 
Moreover, it lifts to an isomorphism of universal cameral curves. 
We therefore make no notational distinction between the Hitchin base and the cameral curves for $G$ and $^L G$ in the following. 

The next corollary is a special case of \cite[Theorem A (2)]{DP}. 
\begin{cor}\label{cor:langlandsadsc}
Let $G=G_{ad}$ or $G_{sc}$. 
Then there is a canonical isomorphism 
\begin{equation*}
P_{^L G}\cong P_G^\vee.
\end{equation*}
\end{cor}
\begin{proof}
It is sufficient to consider the case $P=P_{ad}=P_{ad}^\circ$.
Let $G=G_{ad}$ be the simple adjoint complex Lie group with Dynkin diagram $\Delta$ so that $^L G=G_{sc}(^L \Delta)$. 
Since $\bLambda_G^\vee=\bLambda_{^L G}$, Corollary \ref{cor:Hodge} and Proposition \ref{p:cochar} imply 
\begin{align*}
\mathsf{cochar}(^L P_G)=H^1(\cu, p_{b,*}^W \bLambda_{^L G}^\vee)^\vee=  H^1(\cu, p_{b,*}^W\bLambda_G)^\vee=\mathsf{cochar}(P_G)^\vee.
\end{align*}
\end{proof}


\subsection{Relation to Zucker's Hodge structure}\label{ss:relzucker}
In Section \ref{SectIsoHit} we need an alternative construction of the Hodge structure on $H^1(\cu,p^W_{b,*}\bLambda)_{\mathrm{tf}}$, $b\in \Bbo$
This construction is based on Zucker's results \cite{Zucker}. 
The general setup is the following: 

Let $C$ be any compact Riemann surface and $j:C^\circ\hookrightarrow C$ the complement of finitely many points. 
If $\VH$ is a polarizable $\Z$-VHS of weight $m$ on $C^\circ$, then Zucker (\cite[Theorem 7.12]{Zucker}) constructed a  polarizable $\Z$-Hodge structure of weight $k+m$ on the sheaf cohomology groups $H^k(C,j_*\VH)_{\mathrm{tf}}$, $k=0,1,2$. 
It is compatible with Tate twists and functorial with respect to morphisms of Riemann surfaces and morphisms of $\Z$-VHS. 

The next lemma determines the Hodge structure if $\VH$ is of Tate type.

\begin{lem}\label{ApplicationZuckersVHS}
	Let $j:C^\circ \hookrightarrow C$ be as before and $\VH$ a polarized $\Z$-VHS of weight $m=2k$ and Tate type over $C^\circ$. 		
	Then there exists a commutative diagram 
	\begin{equation}\label{LemZuckersVHS}
	\begin{tikzcd}
	\hat{C}^\circ \arrow[r, hookrightarrow, "\hat{j}"]  \arrow[d, "f^\circ"'] & \hat{C} \arrow[d, "f"]  \\
 	C^\circ \arrow[r, hookrightarrow, "j"] & C
	\end{tikzcd}
	\end{equation}
	where $f$ is a branched Galois covering and $C^\circ$ the complement of its branch divisor. 
	Zucker's Hodge structure on $H^1(\Sigma, j_*\VH)_{\tf}$ is isogenous to $H^1(\hat{\Sigma}^\circ, \hat{j}_*\VH_0)^W=H^1(\hat{\Sigma},\VH_0)^W$ where $W$ is the covering group of $f^\circ$ and $\VH_0$ the typical stalk of $\VH_\Z$. 
	In particular, $H^1(\Sigma,j_*\VH)_{\tf}$ only has types $(k+1,k)$ and $(k,k+1)$. 
\end{lem}
\begin{proof}
	Up to a Tate twist, the $\Z$-VHS $\VH$ consists of a local system $\VH_{\Z}$ of positive definite lattices so that we only write $\VH=\VH_{\Z}$. 
	This implies that its monodromy group $W$ has to be finite and we obtain an unbranched Galois covering $f^\circ: \hat{C}^\circ \to C^\circ$ with covering group $W$. 
	Since $f^\circ$ is locally given by $z\mapsto z^k$, it uniquely completes to a branched covering $f:\hat{C}\to C$ yielding (\ref{LemZuckersVHS}). \\
	By construction $(f^\circ)^*\VH\cong \VH_0$ so that $\VH\cong (f^\circ_*\VH_0)^W$. 
	The inclusion $i:(f^\circ_*\VH_0)^W\hookrightarrow f^\circ_*\VH_0$ is obviously a morphism of VHS.
	Moreover, the natural morphism 
	\begin{equation*}
	\phi:H^1(C^\circ, f^\circ_*\VH_0)_{\tf}\to H^1(\hat{C}^\circ, \VH_0)_{\tf},
	\end{equation*}
	 induced by the Leray spectral sequence, is a morphism of Hodge structures. 
	As $f^\circ$ is finite, $\phi$ is an isomorphism. 
	By the $W$-equivariance of $f^\circ$, these morphisms fit into the commutative diagram 
	\begin{equation}\label{ZuckerHodgeDiagram}
	\begin{tikzcd}
	H^1(C^\circ, f_*^\circ \VH_0) \ar[r, "\phi"] & H^1(\hat{C}^\circ, \VH_0) \\
	H^1(C^\circ, \VH) \ar[r, "\phi^W"] \ar[u, "i"] & H^1(\hat{C}^\circ, \VH_0)^W \ar[u, hookrightarrow] & & \VH=(f_*^\circ \VH_0)^W \\
	H^1_c(C^\circ, \VH) \ar[r, "\psi^W"] \ar[u] & H^1_c(\hat{C}^\circ, \VH_0)^W.\ar[u]
	\end{tikzcd}
	\end{equation}
	Here $\psi^W$ is the natural morphism $H^1_c(C^\circ, \VH)\to H^1_c(\hat{C}^\circ, \VH_0)^W$.
	Arguing as above, we see that it is compatible with mixed Hodge structures. 
	Thus (\ref{ZuckerHodgeDiagram}) is a commutative diagram of mixed Hodge structures. 
	Further $\psi^W$ and $\phi^W$ are isomorphisms over $\Q$ because we can then split off $(f^\circ_*\VH)^W$.
	But the lower square in (\ref{ZuckerHodgeDiagram}) factorizes over 
	\begin{equation*}
	H^1(C, j_*\VH)\to H^1(\hat{C}^\circ, j_*\VH_{0})^W=H^1(\hat{C}, \VH_{0})^W
	\end{equation*}
	which thus has to be an isogeny as well. 
\end{proof}
This lemma fits precisely in the previous situation: 
Let $j:\cuo\hookrightarrow \cu$ be the complement of the branch divisor of the cameral curve $p_b:\tcu_b\to \cu$ for $b\in \Bbo(\cu,G)$ and any simple complex Lie group $G$. 
Then $\VH=(p_b^\circ)^W_{*}\bLambda_G$ is a $\Z$-VHS of weight $0$ and Tate type. 
Moreover, the adjunction morphism 
\begin{equation*}
p_{b,*}^W \bLambda_G\to j_*j^*p_{b,*}^W\bLambda_G =j_*\VH
\end{equation*}
is an isomorphism. 
Hence $H^1(\cu,p_{b,*}^W\bLambda)_{\mathrm{tf}}$ is endowed with Zucker's Hodge structure of weight $1$. 
By Lemma \ref{ApplicationZuckersVHS} and Corollary \ref{cor:Hodge}, it coincides with the Hodge structure $H_G^\circ$ determined by $P_G^\circ(b)$. 
Replacing $\bLambda_G$ with $\bLambda_G^\vee$ and taking duals, we obtain the same statement for $\overline{H}_G$ and $\overline{P}_G(b)$.

Summarizing we obtain:
\begin{prop}\label{eq:relationToZucker}
	The polarizable $\Z$-Hodge structures $H^\circ_G$ and $\overline{H}_G$ of weight $1$ on $H^1(\cu,p_{b,*}^W\bLambda_G)_{\mathrm{tf}}$ and $H^1(\cu, p_{b,*}^W\bLambda_G^\vee)^\vee$ respectively coincide with Zucker's Hodge structures. 
	
	In particular, if $G$ is either a simple adjoint or simply-connected complex Lie group, then the Hodge structure $H_G$ with $J(H_G)\cong P_G$ is determined by Zucker's Hodge structure, cf. Remark \ref{rem:hodgep}. 
\end{prop}

\subsection{VHS and abstract Seiberg--Witten differential}
We next determine the holomorphic symplectic form $\omega_H$ on $\Hig^\circ(\cu, G)$ in terms of an abstract Seiberg--Witten differential. 

Let $\VH_G=\VH(\Hit^\circ_G)$ be the $\Z$-VHS determined by $\Hit^\circ:\Hig^\circ(\cu, G) \to \Bbo$. 
Lemma \ref{PrymsComplexStr} implies 
$
\VH_{G,\C}\cong (R^1\bm{p}_*\otimes \tfr)^W.
$
Hence the section
\begin{equation}\label{SWdifferential}
\bm{\lambda}:\Bbo\to F^1 \VH_{G,\Oo},\quad \bm{\lambda}(b)=\lambda_b \in H^0(\tcu_b,K_{\tcu}\otimes \tfr)^W
\end{equation}
is well-defined.

\begin{prop}\label{SWHitchin}
	The section $\bm{\lambda}\in H^0(\Bbo, \VH_G)$ is an abstract Seiberg--Witten differential. 
	Any Lagrangian section $s$ of $\Hig^\circ(\cu,G)$ induces an isomorphism 
	\begin{equation}\label{eq:isoabswhit}
	(\Hig^\circ(\cu,G),\omega_H)\cong (J(\VH_G),\omega_{\bm{\lambda}} )
	\end{equation} 
	as algebraic integrable systems over $\Bbo$ which sends $s$ to the zero section. 
\end{prop}
\begin{proof}
	It is proven in \cite[Proposition 8.2.]{HHP} that $\bm{\lambda}$ is an abstract Seiberg--Witten differential, i.e. 
	\begin{equation*}
	\phi_{\bm{\lambda}}:T\Bbo\to F^1\VH_G,\quad
	X\mapsto \nabla_X \bm{\lambda}, 
	\end{equation*}
	is an isomorphism. 
	Hence $J(\VH_G)\to \Bbo$ carries a Lagrangian structure $\omega_{\bm{\lambda}}$ by Proposition \ref{CorAbstractSWdifferential} where we use the natural polarization on $\VH_G$. 
	
	Since $\Hig^\circ(\cu,G)$ is a torsor for $J(\VH_G)$ over $\Bbo$ and Lagrangian sections of $\Hit$ exist, e.g. Hitchin sections, $\Hig^\circ(\cu,G)$ is isomorphic to $J(\VH_G)$ over $\Bbo$. 
	We fix such a section from now on. 
	It induces the symplectomorphism 
	\begin{equation*}
	(\Hig^\circ(\cu,G),\omega_H)\cong (T^*\Bbo/\Gamma, \hat{\eta}).
	\end{equation*}
	Here $\Gamma\subset T^*\Bbo$ is a fiberwise lattice determined by the Hamiltonian flows along the fibers (see \cite[Section 2]{Beck-thesis} for details) and $\hat{\eta}$ is the tautological symplectic structure on $T^*\Bbo$ descended to $T^*\Bbo/\Gamma$. 
	
	On the other hand, $\omega_{\bm{\lambda}}$ is determined by a symplectomorphism
	\begin{equation*}
	(T^*\Bbo/\Gamma, \hat{\eta})\cong (J(\VH_G), \omega_{\bm{\lambda}}). 
	\end{equation*}
	Combinining these two symplectomorphisms, we arrive at (\ref{eq:isoabswhit}).
\end{proof}

For the next corollary, define the isogenous $\Z$-VHS $\VH_G^\circ$ and $\overline{\VH}_G$ of weight $1$ with underlying $\Z$-local systems 
\begin{equation}\label{eq:VoG}
\VH_{G,\Z}^\circ=(R^1\bm{p}_{1,*}\bm{p}_{2,*}^W\bLambda_G)_{\tf} , \quad \overline{\VH}_{G,\Z}=(R^1\bm{p}_{1,*}\bm{p}_{2,*}^W\bLambda_G^\vee)^\vee. 
\end{equation}
By the previous proposition, $J(\VH^\circ_G)$ and $J(\overline{\VH}_G)$ are algebraic integrable systems over $\Bbo$.

\begin{cor}\label{cor:langlands}
	Let $G=G_{ad}$ or $G_{sc}$. 
	Then there is an isomorphism 
	\begin{equation*}
	\bm{\ell}:\Hig^\circ(\cu,G)^\vee\cong \Hig^\circ(\cu,^L G)
	\end{equation*}
	of algebraic integrable systems over $\Bbo$. 
\end{cor}
\begin{proof}
It suffices to consider the case $G=G_{ad}$. 
Then $\VH^\circ_G=\VH_G$ by Remark \ref{rem:hodgep}. 
Corollary \ref{cor:langlandsadsc} globalizes to give $\VH_G^\vee\cong \VH_{^L G}=\overline{\VH}_{^L G}$.
Now the statement follows from Proposition \ref{SWHitchin}.
\end{proof}
	This corollary is a special case of \cite[Theorem B]{DP}. 
	Our proof is Hodge-theoretic whereas the one in loc. cit. is moduli-theoretic. 
	With some extra work our method of proof generalizes to other $G$ but we only need $G=G_{ad}$ or $G=G_{sc}$ in the following. 

\section{Non-compact Calabi--Yau integrable systems}\label{SectNcCY3s}
In this section, we construct \emph{algebraic} integrable systems from certain families $\X\to \Bb$ of quasi-projective Gorenstein Calabi--Yau threefolds over the Hitchin base $\Bb=\Bb(\cu, G)$ of a simple complex Lie group $G$ of adjoint type $(G=G_{ad}$).
We call them \emph{non-compact Calabi--Yau integrable systems} which (implicitly) first appeared in \cite{DDP} (also see \cite[Section 7]{KS1}). 

\subsection{Slodowy slices}\label{SectSlodowySlices}
Let $\Delta=\Delta_h^\Cc$ be an irreducible Dynkin diagram where $\Delta_h$ is the corresponding irreducible $\ADE$-Dynkin diagram with associated symmetry group $\Cc=\Cc(\Delta)\subset \Aut(\Delta_h)$, cf. Appendix \ref{SectFolding}. 
A $\Delta$-singularity $(Y,H)$, see \cite[Section 6.2]{Slo}, consists of a (germ of a) surface singularity $Y=(Y,0)$ of type $\Delta_h$, i.e. an $\ADE$-surface singularity, and a subgroup $H\subset \Aut(Y)$ with the following properties: 
\begin{enumerate}[label=\roman*)]
\item 
$H\cong \Cc$; 
\item
the action of $H$ on $Y-\{0\}$ is free; 
\item\label{eq:iii}
the induced action on the dual resolution graph of the minimal resolution $\hat{Y}\to Y$ coincides with the $\Cc$-action on $\Delta_h$. 
\end{enumerate}
If $\bLambda_{G_h}$ is the cocharacter lattice of the simple adjoint complex Lie group with Dynkin diagram $\Delta_h$, then the only non-trivial cohomology groups of $\hat{Y}$ satisfy
\begin{equation}\label{eq:cohofDeltah}
H^2(\hat{Y},\Z)\cong \bLambda_{G_h}, \quad H^2_c(\hat{Y},\Z)\cong \bLambda_{G_h}^\vee.
\end{equation}
These isomorphisms are compatible with the natural duality pairings. 
By \ref{eq:iii} and folding of Lie groups, see Appendix \ref{SectFolding}, it follows that 
\begin{equation}\label{eq:cohofDelta}
H^2(\hat{Y},\Z)^\Cc\cong \bLambda_{G^\vee} ,\quad H^2_c(\hat{Y},\Z)_{\Cc}\cong \bLambda_{G^\vee}^\vee. 
\end{equation}
Here $G^\vee=G_{ad}(\Delta^\vee)$ is the simple adjoint complex Lie group \emph{with Dynkin diagram $\Delta^\vee$}. 
\begin{rem}\label{rem:folding}
We emphasize that to any $\Delta$-singularity, we associate the simple adjoint complex Lie group $G$ with Dynkin diagram $\Delta^\vee$ via (\ref{eq:cohofDelta}). 
The apparent difference comes from the following: 
In his definition of $\Delta$-singularities, Slodowy uses the folding convention for the cocharacter lattice of $G_h$ which is given by taking \emph{invariants}, see (\ref{eq:cohofDelta}). 
This results in the Dynkin diagram $\Delta=\Delta_h^\Cc$. 
Moreover, we will see below that $\Delta$-singularities are realized inside the simple complex Lie algebra $\gfr(\Delta)$.

On the other hand, the folding procedure of simple complex Lie groups takes \emph{coinvariants} in the character lattices $\Delta_{G_h}^\vee$. 
Since the Dynkin diagram of the folded group $G^\vee$ is by definition the one of its character lattice $\bLambda_G^\vee=\bLambda_{G_h,\Cc}^\vee$, we arrive at the dual Dynkin diagram $\Delta^\vee$. 
\end{rem}

Every $\Delta$-singularity $(Y,H)$ is quasi-homogeneous, in particular, $Y$ carries a $\C^*$-action that commutes with the $H$-action. 
A $\C^*$-deformation of $(Y,H)$ is therefore a $\C^*\times H$-deformation $\mathcal{Y}\to B$ such that $H$ acts trivially on the base. 
As shown in \cite[Section 2]{Slo}, each $(Y,H)$ has a semi-universal $\C^*$-deformation. 
It is realized in the simple complex Lie algebra $\gfr=\gfr(\Delta)$ of type $\Delta$ as follows: 
Let $x\in \gfr$ be a subregular nilpotent element and $(x,y,h)$ an $\mathfrak{sl}_2$-triple in $\gfr$ with semisimple $h$. 
Then the Slodowy slice associated to the triple $(x,y,h)$ is defined by 
\begin{equation*}
S:=x+\ker ad(y)\subset \gfr.
\end{equation*}
It carries a non-trivial action by the group $\C^*\times \Cc$ (cf. \cite[Section 6]{Slo}). 
Here the $\Cc$-action is defined by the action (via adjunction) of the group
\begin{equation*}
C(x,h)/C(x,h)^\circ \cong \Cc
\end{equation*}
which is the group of connected components of $C(x,h)=\{g\in G_{ad}~|~g\cdot x=x, g\cdot h= h\}$. 
For a fixed Cartan subalgebra $\tfr\subset \gfr$, denote by $\chi:\gfr\to \tfr/W$ the adjoint quotient and 
\begin{equation*}
\sigma:=\chi_{|S}: S\to \tfr/W
\end{equation*}
its restriction to the Slodowy slice $S$. 
If we let $\C^*$ act on $\tfr/W$ with \emph{twice} the usual weight and $\Cc$ on $\tfr/W$ trivially, then $\sigma$ is $\C^*\times \Cc$-equivariant. 
Slodowy has shown that the fiber $(\sigma^{-1}(0),\Cc)$ is a $\Delta$-singularity and $\sigma: S\to \tfr/W$ with the $\C^*\times \Cc$-action is a $\C^*$-semi-universal deformation of $(\sigma^{-1}(0),\Cc)$ (\cite{Slo}, Section 8.7).


The stratification of $\tfr^1/W=\tfr^\circ\cup D^s\cup D^l$ introduced in Section \ref{SectStratifications} interacts with the restriction
\begin{equation}\label{SlodowyS1}
S^1:=\sigma^{-1}(\tfr^1/W) \to \tfr^1/W
\end{equation}
of $\sigma:S\to \tfr/W$. 
It coincides with the singularity stratification of $\tfr^1/W$ induced by $\sigma$ meaning that $\sigma$ is smooth over $\tfr^\circ/W$ and 
\begin{equation}\label{ExSingConfg}
\sigma^{-1}(\bar{t})\mbox{ has }
\begin{cases}
\mbox{an }\mathrm{A}_1\mbox{-singularity if }\bar{t}\in D^s,\\
\mbox{an }\mathrm{A}_1\times\mathrm{A}_1\mbox{-singularity if }\bar{t}\in D^l\mbox{ and }\Delta\neq \mathrm{G}_2,\\
\mbox{an }\mathrm{A}_1\times\mathrm{A}_1\times \mathrm{A}_1\mbox{-singularity if }\bar{t}\in D^l\mbox{ and }\Delta=\mathrm{G}_2.
\end{cases}
\end{equation}
Grothendieck's simultaneous resolution\footnote{Technically, this is a simultaneous alteration (in the sense of de Jong) because we have to pass to a (branched) covering. However, we keep the common term simultaneous resolution in the following.} 
restricts to give a simultaneous resolution
\begin{equation*}
\begin{tikzcd}
\tilde{S} \ar[r, "\psi"] \ar[d, "\tilde{\sigma}"] & S \ar[d, "\sigma"] \\
\tfr \ar[r, "q"] & \tfr/W
\end{tikzcd}
\end{equation*}
of $\sigma: S\to \tfr/W$. 
There is a natural $\Cc$-action on $\tilde{S}$ such that all morphisms in this diagram become $\Cc$-equivariant. 
To get $\C^*$-equivariance one has to choose $\tfr$ appropriately (\cite[Remark 1.53]{Beck-thesis}). 

In \cite[Section 2.2]{Beck2}, it is shown that $\psi:\tilde{S}\to S$ is even a simultaneous \emph{symplectic} resolution. 
More precisely, \cite[Proposition 1]{Beck2} shows that the Kostant--Kirillov form on $\gfr$ induces a nowhere-vanishing and $\Cc$-invariant section 
\begin{equation}\label{eq:nu}
\hat{\nu}\in H^0(S,K_\sigma)^\Cc
\end{equation}
for the relative canonical class $K_\sigma$ of $\sigma: S\to \tfr/W$.
In particular, $\psi^*\hat{\nu}_{t}$, $t\in \tfr$, is a holomorphic symplectic form on $\tilde{S}_{t}$ and $\psi_t:\tilde{S}_{t}\to S_{\bar{t}}$ is a symplectic resolution.

\subsection{Construction of threefolds}\label{ss:threefolds}
Fix an irreducible Dynkin diagram $\Delta=\Delta_h^\Cc$ together with a Slodowy slice $S\subset \gfr(\Delta)$ with its $\C^*\times \Cc$-action. 
Then we construct families of surfaces over $\Ub=K_\cu\times_{\C^*} \tfr/W$ as follows: 
Let $L\in \mathrm{Pic}(\cu)$ be a spin bundle, i.e. $L^2=K_\cu$. 
By the $\C^*$-equivariance\footnote{We emphasize again that here $\C^*$ acts with \emph{twice} the standard weights on $\tfr/W$. With these weights on $\tfr/W$, we have $\Ub\cong L\times_{\C^*} \tfr/W$.} of $\sigma:S\to \tfr/W$, we obtain a family 
\begin{equation*}
\bsigma_L: \Sc_L:=L\times_{\C^*} S \to \Ub
\end{equation*}
of surfaces which are $\Cc$-deformations of the $\Delta_h$-singularity. 
Further let $\Bb=\Bb(\cu, G)$ be the Hitchin base for the simple adjoint complex Lie group $G=G_{ad}(\Delta)$. 

Similarly to \cite{Sz1}, \cite{DDP}, we construct a family $\pib_L:\X_L\to \Bb$ of threefolds via the diagram
\begin{equation}\label{familyCY3s}
\begin{tikzcd}
\X_L\arrow[r] \arrow[d, "\pib_{1}"] \arrow[dd, bend right=60, "\pib_L"'] & \Sc_{L} \arrow[d, "\bsigma_{L}"] \\
\Sigma\times \Bb \arrow[d, "\pib_{2}=pr"] \arrow[r, "ev"] & \Ub \\
\Bb.
\end{tikzcd}
\end{equation}
Here the square is cartesian. 
Letting $\Cc$ act trivially on $\Ub$ and $\cu\times \Bb$, all morphisms in the cartesian square of (\ref{familyCY3s}) are $\Cc$-equivariant. 
We denote by $\pi_b:X_b\to \cu$ the restriction of $\pib_1$ to $\cu\times \{b\}$.
\\
Analogously, the simultaneous resolution $\tilde{S}\to \tfr$ can be used to construct a \emph{smooth} family $\tilde{\pib}_L:\Xt_L\to\Bb$ of threefolds that depends on the spin bundle $L$ as well. 
It factorizes over the universal cameral curve by construction, 
\begin{equation*}
\begin{tikzcd}
\Xt_L \ar[rr, bend left,  "\tilde{\pib}_L"] \ar[r, "\tilde{\pib}_1"] & \btSigma \ar[r, "\bm{p}"] & \Bb,
\end{tikzcd}
\end{equation*} 
and each restriction $\tilde{\pi}_b:\tilde{X}_b\to \tSigma_b$ is a simultaneous resolution of $\pi_b:X_b\to \cu$ for $b\in \Bbo$.

\begin{thm}\cite[Theorem 2]{Beck2}\label{ThmCY3s}
Let $\Delta$ be an irreducible Dynkin diagram, $S\subset \gfr(\Delta)$ a Slodowy slice with simultaneous resolution $\tilde{S}$ and $L$ a spin bundle of $\cu$. 
Then $\pib_L:\X_L\to \Bb$ as well as $\tilde{\pib}_L:\Xt_L\to \Bb$ are algebraic families of quasi-projective Gorenstein threefolds with $\Cc$-trivial canonical class. 
The former is smooth over $\Bbo\subset \Bb$ whereas the latter is smooth over all of $\Bb$. 
\end{thm}
\begin{proof}[Sketch of proof]
	For later purposes we briefly indicate the construction of nowhere-vanishing and $\Cc$-invariant sections $s_b\in H^0(X_b,K_{X_b})^\Cc$ (equivalently a $\Cc$-trivialization of $K_{X_b}$) and $\tilde{s}_b\in H^0(\tilde{X}_b,K_{\tilde{X}_b})^\Cc$ for $b\in \Bb$. 
	In \cite[Section 3.3.]{Beck2} we construct nowhere-vanishing and $\Cc$-invariant sections
	\begin{gather}
	\bm{s}\in H^0(\X, K_{\pib_1} \otimes (pr_1\circ \pib_1)^*K_\Sigma), \label{SectionS} \\
	\tilde{\bm{s}}\in H^0(\tilde{\X}, K_{\tilde{\pib}_1}\otimes (pr_1\circ \bm{p}_1)^*K_\cu ), \label{SectionTildeS}
	\end{gather}
	where $pr_1:\cu\times \Bb$ is the projection to the first factor. 
	These are glued from the section $\hat{\nu}\in H^0(S,K_\sigma)^\Cc$, see (\ref{eq:nu}), and its pullback to $\tilde{S}$ respectively.
	Base change and the adjunction formula imply 
	\begin{equation*}
		(K_{\pib_1}\otimes (pr_1\circ \pib_1)^*K_\cu)_{|X_b}\cong K_{\pi_b}\otimes \pi_b^*K_\cu \cong K_{X_b}
	\end{equation*}
	and analogously for $\pib_1$ replaced by $\tilde{\pib}_1$. 
	Therefore the restrictions 
	\begin{equation*}
		s_b:=\bm{s}_{|X_b}\in H^0(X_b, K_{X_b}), \quad \tilde{s}_b:=\tilde{\bm{s}}_{|X_b}\in H^0(X_b,K_{\tilde{X}_b})
	\end{equation*}
	$\Cc$-equivariantly trivialize $K_{X_b}$ and $K_{\tilde{X}_b}$ respectively. 
\end{proof}
From now on we fix a spin bundle $L\in \mathrm{Pic}(\cu)$ and drop it from the notation if unnecessary.
Note that the restriction $\pib^\circ:\X^\circ\to \Bbo$, which is smooth, factors as 
\begin{equation}\label{factorizationpib}
\begin{tikzcd}
\X^\circ\arrow[r] \arrow[d, "\pib_{1}^1"] \arrow[dd, bend right=60, "\pib^\circ"'] & \Sc^1 \arrow[d, "\bsigma^1"] \\
\Sigma\times \Bbo \arrow[d, "\pib_{2}^\circ"] \arrow[r, "ev"] & \Ub^1 \\
\Bbo. 
\end{tikzcd}
\end{equation}
Here $\Sc^1=K_\cu\times_{\C^*} S^1$ for $S^1=\sigma^{-1}(\tfr^1/W)$ as in (\ref{SlodowyS1}). 

\subsection{Non-compact Calabi--Yau integrable systems}\label{ncCYintsys}

Let us fix a Slodowy slice $S=x+\ker \ad(y)\subset \gfr$, a spin bundle $L$ over $\cu$ and a family $\X=\X_{L}\to \Bb$ as in Theorem \ref{ThmCY3s}.
Then the cohomology sheaf 
\begin{equation*}
\VH_\Z^{CY}:=R^3\pib^\circ_*\Z
\end{equation*}
underlies a graded-polarizable $\Z$-VMHS 
\begin{equation*}
\VH^{CY}:=(\VH_\Z^{CY},\mathbb{W}_\bullet^{CY},\F^\bullet_{CY})
\end{equation*}
by \cite[Corollary 1.18]{BrosnanElZein} since $\pib$ is quasi-projective and $R^3\pib^\circ_*\Z$ is locally constant.

\begin{lem}\label{LemMHSofCY3s}
Let $b\in \Bbo$. 
The graded-polarizable $\Z$-mixed Hodge structure on $H^3(X_b,\Z)$ is up to a Tate twist pure of weight $1$ with only possibly non-zero $H^{12}$ and $H^{21}$. 
\end{lem}
Also compare with \cite{DDP} for the case $\Cc\neq 1$. 
\begin{proof}
Consider the Leray spectral sequence for $\pi=\pi_b:X=X_b\to \cu$, 
\begin{equation}\label{LerayH3H1}
H^p(\cu,R^q\pi_*\Z)\Rightarrow H^{p+q}(X,\Z). 
\end{equation}
Let $\cu^\circ\subsetneq \cu$ be the locus of smooth fibers of $\pi$. 
Then $\pi^\circ:X^\circ\to \cu^\circ$ is topologically locally trivial with fiber (diffeomorphic to) the minimal resolution $\tilde{Y}$ of the $\Delta$-singularity $Y$. 
But the latter is homotopic to a bouquet of spheres (\cite[Section 4.3]{Slo2}) so that $(R^q\pi_*\Z)_t=0$ for $q\neq 0,2$.

If $t\in  F:=\cu-\cu^\circ$, choose a small disc $D$ around $t$ such that $D\cap F=\{t\}$.
If $D$ is small enough, we contract $\pi^{-1}(D)$ to the central fiber $Q_t:=\pi^{-1}(t)$ so that $(R^q\pi^\circ_{*}\Z)_t=H^q(Q_t,\Z)$\footnote{Note that $\pi:X\to \cu$ is not proper so that we cannot directly apply base change to conclude the same statement.}.
But $Q_t$ is homeomorphic to $\hat{Y}$ with up to three exceptional curves contracted, see(\ref{ExSingConfg}).
Hence we again conclude $(R^q\pi_{*}\Z)_t=0$ if $q\neq 0,2$. 

Since $\cu$ has cohomological dimension two, it follows that the Leray spectral sequence yields an isomorphism
\begin{equation}\label{eq:leray}
H^3(X,\Z)\cong H^1(\cu, R^2\pi_*\Z).
\end{equation} 
The local system $R^2\pi^\circ_*\Z$ carries a polarizable $\Z$-VHS of weight $2$ and Tate type by \cite[Lemma 1.81]{Beck-thesis}.
If $j:\cu^\circ\to \cu$ is the open inclusion, then $R^2\pi_*\Z\cong j_*R^2\pi^\circ_*\Z$.
Therefore $H^1(\cu,R^2\pi_*\Z)$ is endowed with Zucker's $\Z$-Hodge structure of weight $1+2=3$, see Section \ref{ss:relzucker}.
The Leray spectral sequence for $\pi$ is compatible with the mixed Hodge structures on both sides of \ref{eq:leray} (\cite[Corollary 14.14]{PS-MHS})\footnote{This result uses the theory of mixed Hodge modules. Note that we cannot apply \cite{Zucker} directly because it treats the case of projective morphisms but $\pi$ is quasi-projective.}.
Hence the mixed Hodge structure on $H^3(X,\Z)$ is pure of weight $3$. 
It can be seen as in the proof of Lemma \ref{ApplicationZuckersVHS} that it is effective, i.e. its only (possibly) non-zero  are $H^{12}$ and $H^{21}$. 
\end{proof}
\begin{cor}\label{CorVMHS}
The graded-polarizable $\Z$-VMHS $\VH^{CY}$ is pure of weight $3$, i.e. $\mathbb{W}_{\bullet}^{CY}=0$, and has a second-step Hodge filtration. 
In particular, it is an \emph{admissible} VMHS (\cite{SteenbrinkZucker}). 
\end{cor}
The upshot is that each intermediate Jacobian
\begin{equation*}
J^2(X_b)=H^3(X_b,\C)/(F^2H^3(X_b,\C)+H^3(X_b,\Z)),\quad b\in \Bbo,
\end{equation*}
is an abelian variety in contrast to the case of non-rigid compact Calabi-Yau threefolds. 
Moreover, the intermediate Jacobian fibration  
\begin{equation*}
\begin{tikzcd}
 J^2(\X^\circ /\Bbo):=J(\VH^{CY})\ar[r] & \Bbo
\end{tikzcd} 
\end{equation*}
over $\Bbo$ is a family of abelian varieties.
Here and in the following we often suppress the necessary Tate twist to make $\VH^{CY}$ into a VHS of weight $1$. 
\\
In order to make the relation to $\BCFG$-Hitchin systems, i.e. where the Dynkin diagram of the structure group is of type $\mathrm{B}_k$, $\mathrm{C}_k$, $\mathrm{F}_4$, $\mathrm{G}_2$ ($\BCFG$ for short), we need to consider the $\Cc$-invariants $(\VH^{CY})^\Cc\subset \VH^{CY}$.
This is a polarizable sub-$\Z$-VHS of weight $3$ which again only has a two-step Hodge filtration. 
Hence 
\begin{equation}\label{eq:JC}
\begin{tikzcd}
J_{\Cc}^2(\Xo):=J((\VH^{CY})^{\Cc})\ar[r] & \Bbo
\end{tikzcd}
\end{equation}
is a family of abelian varieties. 

\subsection{Period map and abstract Seiberg--Witten differential}\label{eq:periodmap}
The section $\bm{s}\in H^0(\X, K_{\pib_1} \otimes (\mathrm{pr}_1\circ \pib_1)^*K_\Sigma)$ of (\ref{SectionS}) yields the period map
\begin{equation*}
	\rho_{\bm{s}}:\Bbo \to \VH_{\Oo}^{CY},\quad b\mapsto [\bm{s}_{|X_b} ]
\end{equation*}
on $\Bbo$.
By the $\Cc$-invariance of $\bm{s}$, it maps to the $\Cc$-invariant part of $\VH_{\Oo}^{CY}$.
\begin{prop}\label{p:ncintsys}
	The period map $\rho_{\bm{s}}\in H^0(\Bbo, (\VH_{\Oo}^{CY})^{\Cc})$ is an abstract Seiberg--Witten differential.
	In particular, $J_{\Cc}^2(\X^\circ)\to \Bbo$ carries the structure of an integrable system called \emph{non-compact Calabi--Yau integrable system}.
\end{prop}
\begin{proof}
	We will see a posteriori, see Theorem \ref{thm:main}, that $\rho_{\bm{s}}$ is an abstract Seiberg--Witten differential, i.e. 
	\begin{equation*}
	T\Bbo \to (F^2_{CY})^\Cc,\quad v\mapsto \nabla_v \rho_{\bm{s}},
	\end{equation*}
	is an isomorphism. 
	Hence $J^2_{\Cc}(\X^\circ/\Bbo)\to \Bbo$ carries the structure of an integrable system by Proposition \ref{CorAbstractSWdifferential}.
\end{proof}
An analogous result holds for the family $\tilde{\pib}:\Xt\to \Bb$.
More precisely, the methods of the proof of Lemma \ref{LemMHSofCY3s} show that $R^3\tilde{\pib}_*\Z$ carries the structure of a polarizable $\Z$-VHS which is pure of weight $1$ up to a Tate twist. 
We denote it by 
\begin{equation}\label{VMHSXtilde}
\tilde{\VH}^{CY}=(\tilde{\VH}_{\Z}^{CY},\tilde{F}^\bullet_{CY}).
\end{equation}
Note that we do not have to restrict to $\Bbo$ because $\tilde{\pib}$ is smooth. 

\section{Isomorphism with the Hitchin system}\label{SectIsoHit}
As before let $\Delta=\Delta_{h,\Cc}$ be any irreducible Dynkin diagram\footnote{We exclude the somewhat exceptional case $\Delta=\Delta_h=\mathrm{A}_1$ because it has been extensively treated in \cite{DDD}.}.
We denote by $G=G_{ad}(\Delta)$ the simple adjoint complex Lie group with Dynkin diagram $\Delta$ and by $\gfr$ its Lie algebra.
Let $^L S\subset ^L \gfr$ be a Slodowy slice in the Langlands Lie algebra $^L \gfr$ dual to $\gfr$, i.e. it is a semi-universal deformation of the $\Delta^\vee$-singularity. 
In Section \ref{ss:threefolds} we have constructed a family $\mathcal{X}\to \Bb(\cu, G)$ of quasi-projective Gorenstein threefolds with $\Cc$-trivial canonical class from $^L S$. 
As in Section \ref{ss:hitchinadjsc} we identify $\Bb(\cu,^L G)=\Bb(\cu,G)$ and the corresponding universal cameral curves. 

In the previous section, we have seen that the $\Cc$-invariant intermediate Jacobian fibration 
$
J^2_{\Cc}(\Xo)\to \Bbo
$
carries the structure of an algebraic integrable system. 
It is determined by the $\Z$-VHS $(\VH^{CY})^\Cc$ with the abstract Seiberg--Witten differential $\rho_{\bm{s}}$.
On the other hand, the Hitchin system $\Hit:\Hig^\circ(\cu, G_{ad})\to \Bbo$ over $\Bbo$ is determined by the $\Z$-VHS $\VH_G$ with the abstract Seiberg--Witten differential $\bm{\lambda}$, see Proposition \ref{SWHitchin}.

In this section, we show that 
\begin{equation}\label{eq:isovhsabsw}
((\VH^{CY})^\Cc,\rho_{\bm{s}}) \cong (\VH_G,\bm{\lambda})
\end{equation}
as $\Z$-VHS of weight $1$ with abstract Seiberg--Witten differentials. 
In particular, this implies that $J^2_{\Cc}(\Xo)$ is isomorphic to $\Hig^\circ(\cu,G)$ as smooth algebraic integrable systems over $\Bbo$ by Proposition \ref{CorAbstractSWdifferential}.
The proof of (\ref{eq:isovhsabsw}) is divided into two steps:

\begin{enumerate}[label=\Roman*)]
	\item\label{step1}
	There is an isomorphism $\Psi: \VH_{G,\Z} \to (\VH_{\Z}^{CY})^\Cc$ of $\Z$-local systems. 
	\item\label{step2}
	The isomorphism $\Psi$ respects the Hodge filtrations and $\Psi(\bm{\lambda})=\rho_{\bm{s}}$.
\end{enumerate}

\begin{rem}
\begin{enumerate}[label=\alph*)]
\item
Note that we use $\Delta^\vee$-singularities to realize $G$-Hitchin systems as non-compact Calabi--Yau integrable systems where $G=G_{ad}(\Delta)$. 
This is not mysterious: In Remark \ref{rem:folding} we have seen that the $\Delta^\vee$-singularity yields the cocharacter lattice $\bLambda_{G}$ of $G$ via its minimal resolution. 
In turn $\bLambda_G$ is a crucial input to obtain the generic Hitchin fibers for $G$, see Corollary \ref{cor:Hodge}. 
\item 
The step \ref{step1} and the first half of \ref{step2} can be completed in one go by working with mixed Hodge modules. 
These lift the Leray spectral sequence, essential for step \ref{step1}, to a spectral sequence of variations of mixed Hodge structures. 
Since it is more technical, we do not further pursue it here but refer to \cite[Section 5.2]{Beck-thesis}.
\end{enumerate}
\end{rem}

\subsection{Isomorphism of local systems}
We show step \ref{step1} by using the Leray spectral sequence. 

\begin{lem}\label{LemLeray}
	Let $\pib^\circ=\pib_2^\circ \circ \pib_1^1:\mathcal{X}^\circ\to \Bbo$ be as in (\ref{factorizationpib}). 
	Then the Leray spectral sequence degenerates and gives isomorphisms of abelian sheaves
	\begin{equation}\label{LemLerayEq}
	R^3\pib_*^\circ \Z\cong R^1\pib^\circ_{2*}R^2\pib^1_{1*}\Z, 
	\end{equation}
	Moreover, it yields the natural morphism 
	\begin{equation}\label{Morphism:R1R2C}
	\iota: R^1\pib^\circ_{2*}((R^2\pib^1_{1*}\Z)^\Cc)\rightarrow (R^3\pib_*^\circ \Z)^\Cc.
	\end{equation}
\end{lem}
\begin{proof} 
	We first consider the case without taking $\Cc$-invariants. 
	The Leray spectral sequence for $\pib^\circ=\pib^\circ_2\circ \pib_1^1$ reads as
	\begin{equation*}
	E_2^{p,q}=R^p\pib^\circ_{2*}(R^q\pib^1_{1*}\Z)\Rightarrow R^{p+q}\pib^\circ_*\Z.
	\end{equation*}
	It follows similarly as in the proof of Lemma \ref{LemMHSofCY3s} that $R^q\pib^1_{1*}\Z=0$ for $q\notin\{0,2\}$. 
	This implies that $d_2=0$ on the $E_2$-page. 
	To see that the higher differentials $d_r$, $r\geq 3$, also vanish, observe that the projection $\pib^\circ_2=pr:\Sigma\times \Bbo\to \Bbo$ is proper. Hence for any sheaf $\mathcal{F}$ on $\Sigma\times \Bbo$, the stalks of $R^p\pib^\circ_{2!}\mathcal{F}=R^p\pib^\circ_{2*}\mathcal{F}$ are given by
	\begin{equation*}
	R^p\pib^\circ_{2*}\mathcal{F}_b\cong H^p(\Sigma,i_b^*\mathcal{F}).
	\end{equation*}
	But the cohomological dimension of $\Sigma$ is $2$, so that $R^p\pib^\circ_{2*}\mathcal{F}=0$ for $p>2$. 
	This not only implies $d_r=0$ for $r\geq 3$ but also $R^3\pib^\circ_{2*}(R^0\pib^1_{1*}\Z)=0$. 
	Hence the Leray spectral sequence degenerates and 
	\begin{equation*}
	R^3\pib^\circ_*\Z\cong R^1\pib^\circ_{2*}R^2\pib^1_{1*}\Z.
	\end{equation*}
	The morphism (\ref{Morphism:R1R2C}) is induced by the inclusion $(R^2\pib^\circ_{2*} \Z)^\Cc\hookrightarrow R^2\pib^\circ_{2*}\Z$ together with the compatibility of the Leray spectral sequence with the $\Cc$-action. 
\end{proof}
The next proposition shows that the natural morphism (\ref{Morphism:R1R2C}) is an isomorphism in the case $\Cc\neq 1$ as well. 
It is non-trivial because cohomology does not commute with taking invariants in general. 

\begin{prop}\label{p:iotaiso}
	The natural morphism
	\begin{equation*}
	\iota: R^1\pib^\circ_{2*}((R^2\pib^1_{1*}\Z)^\Cc)\rightarrow (R^3\pib_*^\circ \Z)^\Cc
	\end{equation*}
	of Lemma \ref{LemLeray} is an isomorphism. 
\end{prop}

The proof is postponed to Section \ref{ss:proofprop6}.
Together with the next lemma it establishes the link to the VHS $\VH_G$ determined by the Hitchin system.

\begin{lem}\label{bKeyProp}
	Let $\Ubt^1\subset \Ubt$ and $\Ub^1\subset \Ub$ be as in (\ref{Ubt1}), (\ref{Ub1}) and $\Sc^1:=\bsigma^{-1}(\Ub^1)\subset \Sc$.
	Further denote by $\bm{q}^1:\Ubt^1\to \Ub^1$ and $\bsigma^1:\Sc^1\to \Ub^1$ the restrictions of $\bm{q}$ and $\bsigma$ respectively. 
	Then there are isomorphisms of constructible sheaves
	\begin{equation}\label{bKeyPropEq}
	R^2\bsigma^1_*\Z \cong (\bm{q}^1_*\bLambda_h)^W ,\quad (R^2\bsigma^1_*\Z)^\Cc \cong  (\bm{q}^1_*\bLambda)^W
	\end{equation} 
	Here $\bLambda_{h}=\bLambda_{G_h}$ and $\bLambda=\bLambda_G=\bLambda_{G_h}^\Cc$ for the simple adjoint complex Lie group $G_h=G_{ad}(\Delta_h)$ and $G=G_{ad}(\Delta)$ with Dynkin diagram $\Delta_h$ and $\Delta=\Delta_{h,\Cc}$ respectively. 
\end{lem}

\begin{proof}
	First of all we consider the fiberwise case, i.e. $\sigma^1:S^1\to \tfr^1/W$ and $q^1:\tfr^1\to \tfr^1/W$.
	The monodromy group of the $\Delta$-singularity $((\sigma_1^{-1}(\bar{0},x),\Cc)$ coincides with the  Weyl group $W(\Delta)$, see \cite[Section 4.3]{Slo2}.
	This group is the monodromy group of the local system $R^2\sigma^\circ_*\Z$.
	The stalk of this local system is $\bLambda_h$ so that 
	\begin{equation*}
	R^2\sigma^\circ_*\Z\cong (q_* ^\circ\bLambda_h)^W.
	\end{equation*}
	Since the $W$-action commutes with the $\Cc$-action and $\bLambda_h^\Cc=\bLambda$, see (\ref{eq:cohofDelta}), we obtain the second isomorphism in (\ref{bKeyPropEq}) over $\tfr^\circ/W$.
	
	To obtain the isomorphism over $\tfr^1/W$, let $j:\tfr^\circ/W\hookrightarrow \tfr^1/W$ be the open inclusion. 
	Using the fact that the complement of $\tfr^\circ/W$ in $\tfr^1/W$ is smooth, it is not difficult to check that the adjunction morphism
	\begin{equation*}
	a:\mathcal{F}\to j_*j^*\mathcal{F}
	\end{equation*}
	is an isomorphism for $\mathcal{F}=R^2\sigma^1_*\Z$ or $(q_*\bLambda_h)^W$ (see \cite{Beck-thesis}, Chapter 1.6). 
	\\
	For the global case, note that $\cu$ is covered by open subsets $D\subset \cu$ such that 
	\begin{equation*}
	\begin{tikzcd}
	\Sc_{|D} \ar[d, "\bsigma"'] \ar[r, "\cong"] & D\times S \ar[d, "id\times \sigma"] \\
	\Ub_{|D} \ar[r, "\cong"] \ar[r, "\cong"] & D\times \tfr/W \\
	\Ubt_{|D} \ar[u, "\bm{q}"] \ar[r, "\cong"] & D\times \tfr \ar[u, "id\times q"'].
	\end{tikzcd}
	\end{equation*} 
	The fiberwise considerations imply 
	\begin{equation}\label{LocalIsoSheaves}
	R^2(id\times \sigma^\circ)_*\Z \cong ((id\times q^\circ)_*\bLambda_G)^W. 
	\end{equation}
	These are local models for $R^2\bsigma^\circ_*\Z$ and $(\bm{q}^\circ_*\bLambda)^W$ respectively (over $\Ub_{|D}$). 
	Since $\bsigma: \Sc\to \Ub$ and $\bm{q}:\Ubt\to \Ub$ are glued with the same cocyle (which is uniquely determined by the spin bundle $L$), the isomorphism (\ref{LocalIsoSheaves}) glues to an isomorphism 
	\begin{equation*}
	R^2\bsigma^\circ_*\Z \cong (\bm{q}^\circ_*\bLambda_h)^W.
	\end{equation*}
	By pushing forward with $\bm{j}:\Ub^0\hookrightarrow \Ub^1$ and arguing as in the local case gives (\ref{bKeyPropEq}). 
\end{proof}
\begin{thm}\label{thm:isolocsys}
	The isomorphisms of Lemma (\ref{LemLeray}), Proposition and Lemma \ref{bKeyProp} combine to give an isomorphism
	\begin{equation}\label{eq:psi}
	 \Psi: (\VH_\Z^{CY})^\Cc \rightarrow \VH_{G,\Z}
	\end{equation}
	of $\Z$-local systems over $\Bbo$.
\end{thm}
\begin{proof}
	Since $\pib_1^1:\Xo\to \cu\times\Bbo$ and $\bm{p}_1^1:\bm{\tSigma}^\circ\to \cu\times \Bbo$ is obtained as the pullback of $\bsigma^1:\Sc^1\to \Ub^1$ and $\bm{q}^1:\Ubt^1\to \Ub^1$ respectively, Lemma \ref{bKeyProp} implies 
	\begin{equation*}
	R^2\pib^1_*\Z\cong (	\bm{p}_*^1\bLambda_h)^W,\quad  (R^2\pib^1_*\Z)^\Cc\cong (\bm{p}_{1*}^1\bLambda)^W
	\end{equation*}
	as constructible sheaves. 
	Hence Lemma \ref{LemLeray} yields the morphism
	\begin{equation}
	\begin{tikzcd}
	\Psi^{-1}: \VH_{G,\Z}=R^1\bm{p}_{2*}^\circ(\bm{p}_*^1\bLambda)^W \ar[r,"\cong"]  & R^1\pib^\circ_{2*}(R^2\pib^1_{1*}\Z)^\Cc \ar[r, "\iota"] & (R^3\pib_{*}^\circ\Z)^\Cc=(\VH_\Z^{CY})^\Cc
	\end{tikzcd}
	\end{equation}
	Here $\iota$ is the morphism (\ref{Morphism:R1R2C}) which is an isomorphism by Proposition \ref{p:iotaiso}. 
\end{proof}
The isomorphism $\Psi$ induces the isomorphism
\begin{equation*}
\mathsf{cochar}(J^2_{\Cc}(X_b))=H^3(X_b,\Z)^\Cc\cong H^1(\cu, p_{b,*}^W\bLambda)=\mathsf{cochar}(P_{G}(b))
\end{equation*}
of cocharacter lattices for every $b\in \Bbo$, in particular an isomorphism of real tori. 
We next show that it is even an isomorphism of abelian varieties. 

\subsection{Isomorphism of algebraic integrable systems}\label{ss:isoais}
We next show that the isomorphism $\Psi:(\VH_{\Z}^{CY})^\Cc\to \VH_{G,\Z}$ is an isomorphism of $\Z$-VHS of weight $1$ which intertwines the abstract Seiberg--Witten differentials. 
By Proposition \ref{CorAbstractSWdifferential} it induces an isomorphism of algebraic integrable systems. 

\begin{thm}\label{thm:main}
	Let $G$ be a simple adjoint complex Lie group. 
	The isomorphism $\Psi: (\VH_\Z^{CY})^\Cc \rightarrow \VH_{G,\Z}$ of Theorem \ref{thm:isolocsys} is an isomorphism of polarizable $\Z$-VHS of weight $1$ (up to a Tate twist) such that 
	\begin{equation}\label{eq:psiabsw}
	\Psi(\rho_{\bm{s}})=\bm{\lambda}.
	\end{equation}
	In particular, we obtain an isomorphism 
	\begin{equation*}
	\begin{tikzcd}
	J^2_{\Cc}(\Xo)\ar[r, "\cong"] \ar[r]  & \Hig^\circ(\cu,G) 
	\end{tikzcd}
	\end{equation*}
	of algebraic integrable systems over $\Bbo(\cu,G)$.
\end{thm}
\begin{rem}\label{rem:orbifolds}
	In \cite{Beck2} we consider the family of orbifold stacks $[\X/\Cc]\to \Bb$ associated with $\X\to \Bb$ and its $\Cc$-action. 
	Moreover, we construct the family 
	\begin{equation*}
	J^2([\Xo/\Cc])\to \Bbo 
	\end{equation*}
	of orbifold intermediate Jacobian in terms of the group cohomologies $H^3_\Cc(X_b,\Z)$, $b\in \Bbo$ and show that $J^2([\Xo/\Cc])\cong J^2_{\Cc}(\Xo)$ over $\Bbo$ (\cite[Theorem 6]{Beck2}). 
	In this way $\Cc$-invariants are already `built in'. 
\end{rem}
\begin{proof}	
For each $b\in \Bbo$, the isomorphism $\Psi_b$ is the composition
\begin{equation}
\begin{tikzcd}
H^3(X_b,\Z)^\Cc \ar[r] & H^1(\cu,R^2\pi_{b,*}\Z)^\Cc \ar[r] & H^1(\cu, p_{b,*}^W \bLambda). 
\end{tikzcd}
\end{equation}
These isomorphisms are isomorphisms of Hodge structures if the middle and right-hand term are endowed with Zucker's Hodge structure, cf. the proof of Lemma \ref{LemMHSofCY3s}.
But this Hodge structure on $H^1(\cu, p_{b,*}^W\bLambda)$ coincides with $\VH_{G,b}$ by Proposition \ref{eq:relationToZucker}.
Hence $\Psi_b$ is an isomorphism of Hodge structures. 

Since $\Psi$ is an isomorphism of local systems, $\Psi_b$ is in particular compatible with the monodromy action of $\pi_1(\Bbo,b)$. 
Therefore the Rigidity Theorem for VHS (\cite[Theorem 7.24]{SchmidVHS}) implies that $\Psi$ is an isomorphism of $\Z$-VHS. 

We next prove (\ref{eq:psiabsw}).
Observe that $\tilde{\VH}_{\C}^{CY}$ of (\ref{VMHSXtilde}) is given by
\begin{equation}\label{tildeVHC}
R^3\tilde{\pib}_*\C\cong R^1\bm{p}_*R^2\tilde{\pib}_{1,*} \C\cong R^1\bm{p}_*\tfr_h.
\end{equation}
The first isomorphism follows from the Leray spectral sequence as in Lemma \ref{LemMHSofCY3s}.
The second one is a consequence of the fact that $\tilde{S}\to \tfr$ is $C^\infty$-trivial (\cite[Section 4.2]{Slo2}) and $H^2(\tilde{S}_t,\C)\cong \tfr_h$.
The section $\tilde{\bm{s}}\in H^0(\tilde{\X}, K_{\tilde{\pib}_1}\otimes (\mathrm{pr}_1\circ \bm{p}_1)^*K_\cu )$ from (\ref{SectionTildeS}) is $\Cc$-invariant and induces the period map
\begin{equation*}
\rho_{\tilde{\bm{s}}}:\Bbo\to (\tilde{\VH}^{CY}_{\Oo})^\Cc.
\end{equation*}
It is related to $\rho_{\bm{s}}$ as follows:
The natural map $\tilde{\X}\to \X$ induces a $\Cc$-equivariant morphism 
\begin{equation}\label{eq:Phi*}
\Phi^*:\VH^{CY}_{\C}\to \tilde{\VH}^{CY}_{\C}\cong R^1\bm{p}^\circ_*\tfr.
\end{equation}
It coincides with the natural inclusion $(R^1\bm{p}^\circ_*\tfr)^W\hookrightarrow R^1\bm{p}^\circ_*\tfr$ under the isomorphism $\VH^{CY}_\C\cong (R^1\bm{p}^\circ_*\tfr)^W$.
The last isomorphism in (\ref{eq:Phi*}) is a direct consequence of the fiberwise natural isomorphism $H^1(\tcu_b,\tfr_h)^\Cc\cong H^1(\tcu_b,\tfr)$, $b\in \Bbo$. 

From the construction of $\bm{s}$ and $\tilde{\bm{s}}$, it follows that $\Phi^*\circ \rho_{\bm{s}}=\rho_{\tilde{\bm{s}}}$ (after tensoring with $\Oo_{\Bbo}$).  
Then the equality 
\begin{equation*}
\rho_{\tilde{\bm{s}}}=\bm{\lambda}\in H^0(\Bbo, \VH_G)
\end{equation*}
holds which makes sense because $\VH_{G,\Oo}\subset R^1\bm{p}_*\tfr\otimes \Oo_{\Bbo}$. 
This equality follows from the construction of the Leray spectral sequence for the composition $\tilde{\pib}=\bm{p}\circ \tilde{\pib}_1$ of \emph{submersions} (cf. \cite[Chapter 3.5]{GriffithsHarris}) and the fact that both $\rho_{\tilde{\bm{s}}}$ and $\bm{\lambda}$ are obtained by a base change from the tautological section $\bm{\tau}\in H^0(\Ubt,\tilde{u}^*\Ubt)$.
Translating back to $\VH^{CY}$ gives $\Psi(\rho_{\bm{s}})=\bm{\lambda}$. 
\end{proof}

\subsection{Proof of Proposition \ref{p:iotaiso}}\label{ss:proofprop6}
In this section we prove that the natural morphism
\begin{equation*}
\iota: R^1\pib^\circ_{2*}((R^2\pib^1_{1*}\Z)^\Cc)\rightarrow (R^3\pib_*^\circ \Z)^\Cc, 
\end{equation*}
see (\ref{Morphism:R1R2C}), is an isomorphism. 
Let $b\in \Bbo$ and denote $\pi:=\pi_b:X_b\to \cu$ as well as 
\begin{equation*}
\mathcal{F}:=R^2\pi_*\Z \cong (p_{b,*}^W\bLambda_h), 
\end{equation*}
see Lemma \ref{bKeyProp}. 
Then the map $\iota_b$ induced by $\iota$ between the stalks at $b\in \Bbo$ is the natural morphism 
\begin{equation*}
\begin{tikzcd}
\iota_b:H^1(\cu,\mathcal{F}^\Cc) \ar[r] & H^1(\cu,\mathcal{F})^\Cc \cong H^3(X_b,\Z)^\Cc.
\end{tikzcd}
\end{equation*}
By abuse of notation, we denote the natural morphism $H^1(\cu,\mathcal{F}^\Cc) \to H^1(\cu,\mathcal{F})^\Cc$ by $\iota_b$ as well. 
Hence it remains to show that $\iota_b$ is an isomorphism for each $b\in \Bbo$. 

In order to do so, we explicitly compute $H^1(\cu,\mathcal{F})$: 
Let $j:\cuo\hookrightarrow \cu$ be the open inclusion of the smooth locus of $\pi:X_b\to B$. 
Then $\mathcal{F}=j_*\mathcal{F}^\circ$ for the local system 
\begin{equation*}
\mathcal{F}^\circ=R^2\pi^\circ_*\Z\cong (p_{b*}^\circ\bLambda_h)^W 
\end{equation*}
with finite monodromy. 
This explains the relevance of the following two results. 

\begin{lem}\label{LemjL}
	Let $j:\cuo\hookrightarrow \cu$ be the complement of finitely many points in a compact Riemann surface $\cu$ and $\Lc$ a local system on $\cuo$.
	Then the following holds
	\begin{align*}
	H^1(\cu, j_*\Lc)&\cong \ker[H^1(\cuo, \Lc)\to H^0(\cu, R^1j_*\Lc)] 
	\end{align*}
\end{lem}
\begin{proof}
	The first isomorphism is a consequence of the five-term exact sequence coming from the Leray spectral for the open inclusion $j:\cuo\hookrightarrow \cu$. 
	Its first three (non-trivial) terms are given by 
	\begin{equation}\label{OpenLeray}
	\begin{tikzcd}
	0 \arrow[r] & H^1(\cu, j_*\Lc)  \arrow[r] & H^1(\cuo, \Lc) \arrow[r, "\beta"] & H^0(\cu, R^1j_*\Lc)
	\end{tikzcd}
	\end{equation}
	yielding the first description. 
\end{proof}
\begin{rem}\label{betaCinvariant}
	Clearly, $R^1j_*\Lc$ is a skyscraper sheaf supported on $Br=\cu-\cu^\circ$. 
	If $D_j\subset \cu$ is a small disc around $b_j\in Br$, then a local computation shows that $H^1(D_j,\Lc)=L_{\rho_j}$ are the coinvariants in $L$.
	Taking the limit over all such discs yields
	\begin{equation*}
	R^1j_*\Lc= \bigoplus_{k=1}^m (R^1j_*\Lc)_{y_k}\cong  \bigoplus_{k=1}^m L_{\rho_k}. 
	\end{equation*}
	The morphism $\beta:H^1(\cuo, \Lc)\to \bigoplus_k L_{\rho_k}$ in (\ref{OpenLeray}) associates to a class its values at the stalks. 
	In particular, $\beta$ is $\Cc$-equivariant so that the $\Cc$-action on $H^1(\cu, j_*\Lc)=\ker \beta$ is induced by the one on $H^1(\cuo, \Lc)$. 
\end{rem}
Hence as a first step to compute $H^1(\cu,j_*\mathcal{L})$ for $\mathcal{L}=\mathcal{F}^\circ$ we have to determine $H^1(\cuo, \mathcal{L})$. Let 
\begin{equation*}
\begin{tikzcd}[column sep=scriptsize] \cu-\cuo=\{y_1,\dots, y_n\}\arrow[r, hookrightarrow, "i"] & \cu\end{tikzcd}
\end{equation*} 
be the closed embedding. 
As in \cite{DP} it is convenient to add an extra point $y_0$ to $Br$. 
To give a presentation of $\pi_1(\cuo-\{y_0\}, \mathfrak{o})$ for a base point $\mathfrak{o}$, choose an arc system $\delta_1,\dots, \delta_{2g}, \gamma_0, \gamma_1,\dots, \gamma_m$ where $\delta_1, \dots, \delta_{2g}$ give generators of $\pi_1(\cu, \mathfrak{o})$ and $\gamma_j$ are homotopic to loops around $\mathfrak{o}$. 
Then we have the well-known presentation
\begin{align*}
\pi_1(\cuo-\{y_0\},\mathfrak{o})&=\left\langle \delta_1,\dots, \delta_{2g}, \gamma_0,\dots, \gamma_m~\Bigg|~\gamma_0=\prod_{i=1}^{g} [\delta_i,\delta_{i+g}]\prod_{j=0}^m \gamma_j \right\rangle\\
&=\left\langle \delta_1,\dots, \delta_{2g}, \gamma_1,\dots, \gamma_m\right \rangle,
\end{align*}
We now fix an isomorphism $\Lc_\mathfrak{o}\cong L$ once and for all and denote by $\rho_i=mon(\gamma_i), w_j=mon(\delta_j)\in \Aut(L)$ the monodromy transformation corresponding to $\gamma_i$ and $\delta_j$ respectively. 
Note that $\rho_0=mon(\gamma_0)=id_L$.

The next proposition is a generalization of Proposition 6.5 in \cite{DP} where the case $\Lc=(p^\circ_*\bLambda)^W$ is discussed. 
Its proof is omitted because it works similarly as the on in loc. cit. 
We need a more general statement because we work with $R^2\pi^\circ_*\Z\cong (p^\circ_*\bLambda_h)^W$, i.e. $W$ acts on the larger cocharacter lattice $\bLambda_h=\bLambda_{G_h}$ for $G_h=G_{ad}(\Delta_h)$ not just $\bLambda_G=\bLambda_h^\Cc$.  
\begin{prop}\label{PropDescribingH1}
	Let $p:\tcu\to \cu$ be a smooth cameral curve.
	Further let $\Lc$ be a local system over $\cuo$ with stalk $L\cong \Lc_{\mathfrak{o}}$ and assume $(p^\circ)^*\Lc\cong L_{\tcu^\circ}$.
		Then there is a non-canonical isomorphism 
		\begin{equation*}
		H^1(\cuo-\{y_0\},\Lc)\cong H^1(\cu, L)\oplus \frac{L^{m}}{(1-\rho_1,\dots, 1-\rho_m)L}.
		\end{equation*}
\end{prop}
With these results at hand, we finish the proof of Proposition \ref{p:iotaiso}:
\begin{proof}[End of proof of Proposition \ref{p:iotaiso}]
	Lemma \ref{LemjL} and Remark \ref{betaCinvariant} imply that it is sufficient to show that the map 
	\begin{equation*}
	H^1(\cuo-\{y_0\}, \Lc^\Cc)\to H^1(\cuo-\{y_0\},\Lc)^\Cc\subset H^1(\cuo-\{y_0\}, \Lc),
	\end{equation*}
	induced by the inclusion $\Lc^\Cc=(j^*\F)^\Cc\hookrightarrow \Lc=j^*\F$, is an isomorphism for $\F=R^2\pi_*\Z\cong (p_*\bLambda_h)^W$. 
	By Proposition \ref{PropDescribingH1} this amounts to showing that the natural map 
	\begin{equation}\label{iotaFactors}
	\tau:
	H^1(\cu, \bLambda)\oplus \frac{\bLambda^m}{(1-\rho_1,\dots, 1-\rho_m)\bLambda} \longrightarrow \left( H^1(\cu, \bLambda_h)\oplus \frac{\bLambda^m_h}{(1-\rho_1,\dots, 1-\rho_m)\bLambda_h}\right)^\Cc
	\end{equation}
	is an isomorphism. 
	Here we have fixed isomorphisms $\Lc_{\mathfrak{o}}\cong \bLambda_h$ and $\Lc^\Cc_{\mathfrak{o}}\cong \bLambda$ as before. 
	Since $p_b:\tcu_b\to \cu$ has only simple ramification, $\rho_j=s_{\alpha_j}\in W$ for roots $\alpha_j$. \\
	Of course, $\tau$ preserves the respective first factors in (\ref{iotaFactors}) giving an isomorphism 
	\begin{equation*}
	H^1(\cu,\bLambda)\cong H^1(\cu, \bLambda_h)^\Cc.
	\end{equation*}
	So it remains to check the second factors. 
	For injectivity, assume $\tau([\lambda_1,\dots, \lambda_m])=0$. 
	This happens iff there exists $\mu\in \bLambda_h$ such that 
	\begin{equation*}
	\lambda_i=(1-\rho_i)\mu=\langle \alpha_i,\mu\rangle  \alpha_i^\vee \in \bLambda\subset \bLambda_h,\quad \forall i=1, \dots, m. 
	\end{equation*}
	We have to exclude the case that $\langle \alpha_i, \mu \rangle \notin \langle \alpha_i, \bLambda\rangle\subset \Z$. 
	However, this is impossible because $1\in \langle \alpha_i, \bLambda\rangle$, since $G$ is adjoint.
	Hence $\tau([\lambda_1,\dots, \lambda_m])=0$ implies $[\lambda_1,\dots, \lambda_m]$.
	
	For the surjectivity of $\tau$, assume $[\lambda_1,\dots, \lambda_m]_h\in \bLambda_h^m/(1-\rho_1,\dots, 1-\rho_m)\bLambda_h$ such that 
	\begin{align*}
	&c \cdot [\lambda_1,\dots, \lambda_m]_h=[\lambda_1,\dots, \lambda_m]_h \\ 
	\Leftrightarrow~ &c\cdot \lambda_i-\lambda_i=\langle \alpha_i,\mu\rangle \alpha_i^\vee, \quad \forall i=1, \dots, m
	\end{align*}
	for some $\mu\in \bLambda_h$ and all $c\in \Cc$. 
	For the moment assume $\Cc=\Z/2\Z$. 
	Then using $\alpha_i^\vee\in \bLambda=\bLambda_h^\Cc$ we have 
	\begin{equation*}
	\sigma\cdot (\sigma\cdot \lambda_i-\lambda_i)=\lambda_i-\sigma\cdot \lambda_i=\sigma\cdot \lambda_i-\lambda_i 
	~\Leftrightarrow~ 2(\sigma\cdot \lambda_i-\lambda_i)=0. 
	\end{equation*}
	Hence $\lambda_i=\sigma\cdot\lambda_i$ for all $i=1,\dots, m$ so that $\lambda_i\in \bLambda_h^\Cc=\bLambda$. 
	In other words, $[\lambda_1,\dots, \lambda_m]_h$ is in the image of $\tau$.
	The case $\Cc=S_3$ works similarly by taking generators of order $2$ and $3$. 
	Therefore $\tau$ is an isomorphism in all cases so that we obtain an isomorphism $H^1(\cu,R^2\pi_*\Z)^\Cc)\cong H^1(\cu,(p_*\bLambda)^W)$ and hence $J^2_\Cc(X_b)\cong P_b$ as \emph{real} tori. 
	
	The functoriality of Zucker's Hodge structure applied to the inclusion $\Lc^\Cc\hookrightarrow \Lc$ (of polarized $\Z$-VHS of weight $2$ over $\cuo$) implies that the induced monomorphism from above, 
	\begin{equation*}
	\begin{tikzcd}
	H^1(\cu, (j_*\Lc)^\Cc)\arrow[r] &  H^1(\cu, j_*\Lc)^\Cc
	\end{tikzcd}
	\end{equation*}
	is a monomorphism of polarized $\Z$-Hodge structures of weight $2+1=3$. 
	Therefore
	\begin{equation*}
	H^1(\cu,(p_*\bLambda)^W)\cong H^1(\cu, (R^2\pi_*\Z)^\Cc)\cong H^1(\cu, R^2\pi_*\Z)^\Cc
	\end{equation*}
	as polarizable $\Z$-Hodge structures of weight $1$ showing that $J^2_{\Cc}(X_b) \cong P_b$ as abelian varieties. 
\end{proof}

\section{Poincar\'{e}-Verdier and Langlands duality}\label{bLanglandsDual}
Langlands duality of Hitchin systems (\cite{DP}) states, among other things, that every $G$-Hitchin system is the dual torus fibration of the corresponding $^L G$-Hitchin system over the smooth locus, cf. Corollary \ref{cor:langlandsadsc}.
In this section, we show that this statement is equivalent to Poincar\'{e}-Verdier duality for the families of quasi-projective Calabi--Yau threefolds of Theorem \ref{ThmCY3s} if the simple adjoint complex Lie group $G=G_{ad}(\Delta)$ does not have Dynkin diagram $\Delta=\mathrm{B}_k$.
Since many of the arguments work analogously as in Section \ref{SectIsoHit}, we are brief and omit details.

\subsection{Homology intermediate Jacobian}
As before let $\Delta=\Delta_{h,\Cc}$ be any irreducible Dynkin diagram and $G=G_{ad}(\Delta)$ the corresponding simple adjoint complex Lie group. 
Further let $\pib:\X\to \Bb=\Bb(\cu,G)$ be a family of Gorenstein threefolds with $\Cc$-trivial canonical class constructed from a $\Delta^\vee$-singularity, see Theorem \ref{ThmCY3s}.
If $b\in \Bbo$, then Poincar\'{e} duality implies 
\begin{equation}\label{eq:poincareXb}
(H^3(X_b,\Q))^\vee:=\mathrm{Hom}_{MHS}(H^3(X_b,\Q),\Q(-3)) \cong H^3_c(X_b,\Q)
\end{equation}
as mixed Hodge structures. 
By Lemma \ref{LemMHSofCY3s}, the mixed Hodge structure on $H^3_c(X_b,\Z)$ is a Tate twist of a pure Hodge structure of weight $1$. 
Hence the $\Cc$-coinvariant homology intermediate Jacobian $J_2^\Cc(X_b)$ 
\begin{gather}\label{eq:J_cc2}
J^\Cc_2(X_b):=H^3_c(X_b,\C)_\Cc/(F^2H^3_c(X_b,\C)_\Cc+H^3_c(X_b,\Z)_{\Cc,\tf})
\end{gather}
is an abelian variety. 
The natural map 
\begin{equation}\label{eq:isogenyJ2C}
J_2^\Cc(X_b)\to J^2_{\Cc}(X_b)^\vee
\end{equation}
induced by (\ref{eq:poincareXb}) is an isogeny. 
To obtain a relative statement, observe that $\pib^\circ$ induces a variation of mixed Hodge structures $\VH_{CY}^\Cc$ with underlying local system
\begin{equation*}
\VH_{CY,\Z}^\Cc=(R^3\pib^\circ_! \Z)_{\Cc,\tf}.
\end{equation*}
It is pure of weight $1$ (up to a Tate twist).
Poincar\'{e}-Verdier duality applied to $\pib^\circ$ gives the isogeny
\begin{equation}\label{eq:isogenyj2}
\jmath:J(\VH_{CY}^\Cc)\to J(\VH_\Cc^{CY})^\vee
\end{equation} 
of families of abelian varieties over $\Bbo$, globalizing (\ref{eq:isogenyJ2C}).
In particular, $\VH_{CY}^\Cc$ and $(\VH_\Cc^{CY})^\vee$ are isogenous to each other. 
Hence
\begin{equation*}
J_2^\Cc(\Xo):=J(\VH_{CY}^\Cc)\to \Bbo
\end{equation*}
carries the structure of an algebraic integrable system which is governed by the abstract Seiberg--Witten differential $\rho_{\bm{s}}$, see Section \ref{eq:periodmap}. 

\subsection{Relation to Langlands duality}
For the next statement, recall that we introduced the $\Z$-VHS $\VH_G^\circ$ for any simple complex Lie group in (\ref{eq:VoG}).
If $G=G_{ad}$, then $J(\VH_G^\circ)\cong \Hig^\circ(\cu,G)$ as algebraic integrable systems. 
If $G=G_{sc}$, then $J(\VH_G^\circ)$ is not in general the $G$-Hitchin system. 
However, we next show the relation to non-compact homology Calabi--Yau integrable systems and Langlands duality. 
\begin{thm}\label{thm:langlands}
	Let $G=G_{ad}$ be a simple adjoint complex Lie group. 
	The Leray spectral sequence induces the isomorphism 
	\begin{equation}\label{eq:homj2p}
	J_2^\Cc(\Xo)\cong J(\VH_{^LG}^\circ)
	\end{equation}
	of algebraic integrable systems over $\Bbo(\cu,G)$. 
	In particular, the morphism $\jmath:J_2^\Cc(\Xo)\to J_{\Cc}^2(\Xo)^\vee$ is an isomorphism of algebraic integrable systems over $\Bbo$ if the Dynkin diagram $\Delta$ of $G$ is not of type $\mathrm{B}_k$.  
	In these cases we further obtain the isomorphism
	\begin{equation}
	\begin{tikzcd}
	J_2^\Cc(\Xo)\ar[r, "\cong"] & \Hig^\circ(\cu,^LG)
	\end{tikzcd}
	\end{equation}
	of algebraic integrable systems over $\Bbo$. 
\end{thm}

\begin{proof}
We focus on the fiberwise isomorphism. 
The global isomorphism works as in Section 5. 
The Leray spectral sequence for compactly supported cohomology, the analogue of Lemma \ref{bKeyProp} together with (\ref{eq:cohofDeltah}) give the isomorphism
\begin{equation}
H^3_c(X_b,\Z)\cong H^1(\cu, p_{b,*}^W \bLambda^\vee_{G_h}).
\end{equation}
Recall that the $\Cc$-coinvariants satisfy $\bLambda^\vee_{G_h,\Cc}\cong \bLambda^\vee_{G}=\bLambda_{^L G}$.
As in the case of $\Cc$-invariants, see Section \ref{ss:proofprop6}, one proves that 
\begin{equation*}
H^1(\cu, p_{b,*}^W\bLambda^\vee_{G_h})_{\Cc}\cong H^1(\cu,p_{b,*}^W\bLambda^\vee_{G_h,\Cc} )=H^1(\cu, p_{b,*}^W\bLambda_{^L G}).
\end{equation*}
In particular, $\mathrm{cochar}(J_2^\Cc(X_b))\cong \mathrm{cochar}(P_{^L G}^\circ(b))$ and as in the proof of Theorem \ref{thm:main} it follows that $J_2^\Cc(X_b)\cong P_{^L G}^\circ$ as abelian varieties. 

For the last two statements, recall that $J^2_\Cc(X_b)\cong P_{G}(b)$ and $P_{^L G}(b)=P^\circ_{^L G}(b)$ if $\Delta\neq \mathrm{B}_k$, see (\ref{eq:allIso}). 
In these cases, the morphism $\jmath_b:J_{\Cc}^2(X_b)\to J_2^\Cc(X_b)^\vee$ coincides with the canonical isomorphism $P_G(b)\cong P_{^L G}^\vee$ of Corollary \ref{cor:langlandsadsc}.
\end{proof}
\begin{rem}\label{rem:exceptional}
	In particular, (the smooth locus of) every $G$-Hitchin system, where $G$ is a simple adjoint or simply-connected complex Lie group, is isomorphic to a non-compact cohomology or homology Calabi--Yau integrable system except for $G=Sp(r,\C)$.
	In this case the Hitchin fiber $P_G(b)$, $b\in \Bbo(\cu, G)$, is $\overline{P}_G(b)\neq P^\circ_G(b)$, see Proposition \ref{p:cochar}.
	However, the algebraic integrable system $J_2^\Cc(\Xo)\to \Bbo(\cu, G)$ gives a geometric meaning to the fibers $P^\circ_{G}(b)$ for $G=Sp(n,\C)$.
\end{rem}

If $G=G_{ad}(\Delta)$ is the simple adjoint complex Lie group $G$ with Dynkin diagram $\Delta$, the previous theorem and Theorem \ref{thm:main} yield the commutative diagram
\begin{equation*}
\begin{tikzcd}
	J_2^\Cc(\Xo) \ar[r, "\sim "] \ar[d, "\jmath"', "\sim" labln ] & \Hig^\circ(\cu, ^L G)  \ar[d, "\bm{\ell}", "\cong"' labls] 
	\\
	J^2_{\Cc}(\Xo)^\vee \ar[r, "\cong"] & \Hig^\circ(\cu, G)^\vee 
\end{tikzcd}
\end{equation*}
of algebraic integrable systems over $\Bbo(\cu,G)$.
Here $\sim$ stands for isogenous algebraic integrable systems. 
If $\Delta\neq \mathrm{B}_k$, then these arrows are isomorphisms. 
In that sense, Poincar\'{e}-Verdier duality applied to the family $\Xo\to \Bbo$ recovers the Langlands duality statement of Corollary \ref{cor:langlands}.

\appendix
\section{Folding}\label{SectFolding}
Let $\Delta$ be an irreducible Dynkin diagram. 
We define the associated symmetry group of $\Delta$ by
\begin{equation}\label{ASDelta}
\Cc:=\Cc(\Delta):=
\begin{cases}
1, & \Delta\mbox{ is of type }\ADE, \\
\Z/2\Z, & \Delta\mbox{ is of type }\mathrm{B}_{k}, \mathrm{C}_{k}, \mathrm{F}_4, \\
S_3, & \Delta\mbox{ is of type }\mathrm{G}_2,
\end{cases}
\end{equation}
for $k\geq 2$. 
There is a unique irreducible $\ADE$-Dynkin diagram $\Delta_h$ with $\Cc\subset \Aut(\Delta_h)$ and $\Delta=\Delta_h^\Cc$.
Here $\Delta_h^\Cc$ stands for the Dynkin diagram which is obtained by taking $\Cc$-invariants in the root space associated with $\Delta_h$. 
Dually, taking $\Cc$-coinvariants gives the Dynkin diagram $\Delta^\vee=\Delta_{h,\Cc}$. 
Both procedures are called \emph{folding}, cf. Figure \ref{Figure1}.

On the level of simple complex Lie groups, folding goes as follows (\cite[Chapter 10.3.]{Springer}): 
Let $G_h$ be a simple complex Lie group with character and cocharacter lattice $\bLambda_{G_h}^\vee\subset \mathsf{weight}_{\gfr_h}$ and $\bLambda_{G_h}\subset \mathsf{coweight}_{\gfr_h}$ respectively.
The subgroup $\Cc\subset \mathrm{Aut}(\Delta_h)$ yields the outer automorphism group $\Cc\subset \mathrm{Aut}(G_h)$. 
Then the folded group $G:=(G_h^\Cc)^\circ$, i.e. the connected component of the fixed point group $G_h^\Cc$, is uniquely determined by the character and cocharacter lattice
\begin{equation*}
\bLambda_G^\vee=\bLambda_{G_h,\Cc},\quad \bLambda_G=\bLambda_{G_h}^\Cc.
\end{equation*}
Note that the Dynkin diagram of $\bLambda_G^\vee$, hence the one of $G$, is $\Delta^\vee=\Delta_{h,\Cc}$ whereas $\bLambda_G$ has Dynkin diagram $\Delta=\Delta_{h,\Cc}$. 

For convenience, we summarize how all irreducible Dynkin diagram of type $\mathrm{B}_k$, $\mathrm{C}_k$, $\mathrm{F}_4$, $\mathrm{G}_2$ (\emph{$\BCFG$-Dynkin diagrams for short}) are obtained via folding:

\begin{equation}\label{FoldingDynkin}
\begin{array}{c|c|c|c} 
\Delta_h & \Cc & \Delta^\vee=\Delta_{h,\Cc} & \Delta=\Delta_h^\Cc   \\ \hline
 \mathrm{A}_{2k-1} & \Z/2\Z & \mathrm{C}_{k} & \mathrm{B_{k}}   \\ 
 \mathrm{D}_{k+1} &   \Z/2\Z & \mathrm{B}_{k} & \mathrm{C}_{k}    \\ 
\mathrm{E}_6  & \Z/2\Z & \mathrm{F}_4 & \mathrm{F}_4   \\ 
 \mathrm{D}_4 & S_3 & \mathrm{G}_2 & \mathrm{G}_2\\ 
\end{array}
\end{equation}

\bibliographystyle{plain}

\bibliography{bibtex1}

\end{document}